\documentclass[openany]{amsart}
\usepackage{amssymb, amsfonts, lacromay,amsmath}
\usepackage{mathrsfs,comment}
\usepackage[normalem]{ulem}

\usepackage{url}

\usepackage[all,arc,2cell]{xy}
\UseAllTwocells
\usepackage{enumerate}
\usepackage{mathrsfs}

\newtheorem{thm}{Theorem}[section]
\newtheorem{cor}[thm]{Corollary}
\newtheorem{prop}[thm]{Proposition}
\newtheorem{lem}[thm]{Lemma}

\theoremstyle{definition}
\newtheorem{defn}[thm]{Definition}

\newtheorem{con}[thm]{Construction}

\newtheorem{notn}[thm]{Notation}

\newtheorem{rem}[thm]{Remark}

\makeatletter
\let\c@equation\c@thm
\makeatother
\numberwithin{equation}{section}

\makeatletter
\ifx\SK@label\undefined\let\SK@label\label\fi
 \let\your@thm\@thm
 \def\@thm#1#2#3{\gdef\currthmtype{#3}\your@thm{#1}{#2}{#3}}
 \def\mylabel#1{{\let\your@currentlabel\@currentlabel\def\@currentlabel
  {\currthmtype~\your@currentlabel}
 \SK@label{#1@}}\label{#1}}
 \def\myref#1{\ref{#1@}}
\makeatother

\newcommand{\mb}[1]{\mathbf{#1}}

\newcommand{\Fun}{\mathrm{Fun}}

\newcommand{\gen}{$\bF_\bullet$}

\newcommand{\xrtarr}[1]{\xrightarrow{#1}}

\newcommand{\sbt}{{\, \begin{picture}(-1,1)(0.5,-1)\circle*{1.8}\end{picture}\hspace{.05cm}}}

\newcommand{\shim}{\bS_G}
\newcommand{\shimsma}{\bS_G^{\bN_G}}
\newcommand{\shimtimes}{ \widetilde{\bS}_G^{\bN_G}}
\newcommand{\symfsma}{\bS_G^{\SI}}

\newcommand{\symsma}{\bS_G^{\SI_G}}

\newcommand{\barc}[3]{B({#1},{#2},{#3})}

\title[A symmetric monoidal  equivariant Segal infinite loop space machine]{A symmetric monoidal and equivariant Segal infinite loop space machine}

\author{B. Guillou}
\address{Department of Mathematics, The University of Kentucky, Lexington, KY 40506--0027}
\email{bertguillou@uky.edu}
\author{J.P. May}
\address{Department of Mathematics, The University of Chicago, Chicago, IL 60637}
\email{may@math.uchicago.edu}
\author{M. Merling}
\address{Department of Mathematics, University of Pennsylvania, Philadelphia, 19104}
\email{mmerling@math.upenn.edu}
\author{A.M. Osorno}
\address{Department of Mathematics, Reed College, Portland, OR 97202}
\email{aosorno@reed.edu}

\thanks{B. Guillou was  supported  by Simons Collaboration Grant No. 282316 and NSF grant DMS-1710379.}\thanks{M. Merling was supported by NSF grant DMS-1709461}
\thanks{A.M. Osorno was  supported by Simons Foundation Grant No. 359449, the Woodrow Wilson Career Enhancement Fellowship and NSF grant DMS-1709302}
\thanks{NSF RTG grant DMS-1344997 supported several collaborator visits to Chicago.}



\begin{document}

\begin{abstract}
In \cite{MMO}, we reworked and generalized equivariant infinite loop space theory, which shows
how to construct $G$-spectra from $G$-spaces with suitable structure. In this paper, we construct a new variant of the equivariant  Segal machine that starts  from the category $\sF$ of finite sets rather than from  the category $\sF_G$ of finite $G$-sets and which is equivalent to the machine  studied in  \cite{Shim, MMO}. In contrast to the machine in \cite{Shim, MMO},  the new machine gives a lax symmetric monoidal functor from the symmetric monoidal category of $\sF$-$G$-spaces  to the symmetric  monoidal category of orthogonal $G$-spectra. We relate it multiplicatively to suspension $G$-spectra and to Eilenberg-Mac\,Lane $G$-spectra via lax symmetric monoidal functors from based $G$-spaces and from abelian groups to $\sF$-$G$-spaces. 
Even non-equivariantly, this gives an appealing new variant of the Segal machine.   This new variant makes the equivariant generalization of the theory essentially formal, hence  likely to be applicable in other contexts.
\end{abstract}

\maketitle

\tableofcontents

\section*{Introduction} 

 The Segal infinite loop space machine \cite{Seg} provides one of the two most commonly used approaches for constructing spectra from space level data,
and it is widely used for constructing algebraic $K$-theory spectra.  The input of the Segal infinite loop space machine is $\sF$-spaces, which are functors from the category of based finite sets, which we denote by $\sF$, to the category of spaces.\footnote{Segal defined them as contravariant functors from a certain category $\GA$ and called them $\GA$-spaces, but  $\GA^{op}$ is just the category of finite based sets $\sF$.} The machine constructs spectra from $\sF$-spaces, and all connective spectra arise in this way.
The category of  $\sF$-spaces is
symmetric monoidal under Day convolution, and the conceptual version of the Segal machine from $\sF$-spaces to spectra, defined as a prolongation functor, is lax symmetric monoidal for formal reasons. However, this version of the machine is not homotopically well-behaved except under cofibrancy conditions that are not usually satisfied, and the derived version  given by the bar construction (which can be thought of as a structured cofibrant approximation) loses the symmetry: it is lax monoidal but is \emph{not} lax symmetric monoidal.  

There are two evident equivariant generalizations of nonequivariant $\sF$-spaces: one can consider functors from the category of based finite sets $\sF$ or from the category of based finite $G$-sets $\sF_G$ to $G$-spaces. However, the resulting categories of $\sF$-$G$-spaces and $\sF_G$-$G$-spaces are equivalent \cite{Shim2}. The equivariant Segal machine defined in \cite{Shim} and studied further in \cite{MMO} takes $\sF_G$-$G$-spaces as input. Just as nonequivariantly, this machine is defined as a prolongation functor precomposed with a homotopical replacement of the input, which results in a lax  monoidal but \emph{not} lax symmetric monoidal functor to $G$-spectra.   In \cite{MMO}, this machine is shown to be equivalent to the operadic equivariant infinite loop space machine from \cite{GM3}.

It is desirable to have a machine whose natural input is $\sF$-$G$-spaces rather than $\sF_G$-$G$-spaces.  However, there is a subtle but critical problem with the obvious equivariant
generalization of the homotopical Segal machine that makes it unusable equivariantly. This is explained in
 \cite[Warning 3.10]{MMO}, as we recall in Remark~\ref{why-not-N}: even when fed the correct input $\sF$-$G$-spaces, the usual bar construction defined using $\sF$ fails to produce genuine $\OM$-$G$-spectra.  That is, the usual bar construction on $\sF$-$G$-spaces does not give an equivariant infinite loop space machine when applied to appropriate $\sF$-$G$-spaces, as it does nonequivariantly when applied to $\sF$-spaces \cite{Seg} and equivariantly when applied to $\sF_G$-$G$-spaces \cite{Shim, MMO}.

In this paper, we introduce a modification of the bar construction that solves the problem pointed out in \cite[Warning 3.10]{MMO} and gives an equivariant Segal infinite loop space machine that starts from $\sF$-$G$-spaces. The modification of the bar construction entails complementary use of different monads with the same algebras, as in \cite[\S4]{Rant1}.  Our new version of the Segal machine has the unexpected bonus that it gives a new Segal machine that is both homotopically correct \emph{and} lax symmetric monoidal.

In fact, as summarized in \S\ref{SISIG}, we give several equivalent such machines.   All are equivalent to the
homotopically correct but non-symmetric monoidal machine developed in \cite{Shim, MMO}.
With the new variants of the Segal machine, any multiplicative algebraic structure, such as rings, modules, 
and algebras, that we see on the level of $\sF$-$G$-spaces is automatically transported by the
machine to corresponding algebraic structure on the level of genuine $G$-spectra. 


The new variants of the Segal machine are made possible by an invariance theorem that we prove in \S\ref{inv}.  It allows us to use variant versions of the bar construction and prove that they give equivalent machines. The key construction, developed in \S\ref{formal} and \S\ref{concrete},  is a symmetric version of the bar construction that leads to the good properties of our new machines.  In \S\ref{NNG} we review the classical equivariant Segal machine from \cite{Shim,MMO}. The construction and comparisons of the new machines are given in \S\ref{symma}.  We prove that the new variants are symmetric monoidal in \S\ref{SymMon}.   

We give two
examples in  \S\ref{homotopical}.  One uses a strong symmetric monoidal functor from based $G$-spaces to $\sF$-$G$-spaces to give an equivariant version of the Barratt-Priddy-Quillen theorem that applies the Segal machine to construct suspension $G$-spectra.   Another  uses a lax symmetric monoidal functor from Abelian $G$-groups to $\sF$-$G$-spaces to construct genuine ring, module, and algebra Eilenberg-Mac\,Lane $G$-spectra. Parenthetically, we also  note that the Segal machine preserves homotopies.

\subsection*{Acknowledgements} 
It is a real pleasure to thank Stefan Schwede for very helpful conversations with the third author and, in particular, for suggesting the idea of the proof of the invariance theorem in \S\ref{inv}. It is also a pleasure to thank the referee for a remarkably careful reading that uncovered several glitches.

\section{Definitions, conventions, and the Segal machine}\label{barsection}

We fix a {\bf finite} group $G$ throughout the paper. In this section, we recall some definitions and results from \cite{MMO} 
and specify what we mean by a Segal machine. 

\subsection{The relevant categories of $G$-spaces}\label{sec:CatGSpaces}
All spaces are understood to be compactly generated and weak Hausdorff.  We use the following notation.

\begin{itemize}
\item $\sU$ is the category of unbased spaces and unbased maps; $\sU_*$ and $\sT$ are the categories of based and of nondegenerately based spaces, respectively, and based maps;  
\item  $G\sU$  is the category of unbased $G$-spaces and unbased $G$-maps; $G\sU_*$ and $G\sT$ are the categories of based 
and of nondegenerately based $G$-spaces, respectively, and based $G$-maps; 
\item $\sU_G$ is the category of unbased $G$-spaces and \emph{all} (that is, not necessarily equivariant) unbased maps; ${\sU_G}_*$ and  $\sT_G$ are the categories  of based and of nondegenerately based $G$-spaces, respectively, and \emph{all} based maps. 
\end{itemize}

While $\sU_*$ and its equivariant avatars have better formal properties than $\sT$, all of our 
$\sF$-$G$-spaces and $\sF_G$-$G$-spaces, as defined in \S\ref{sec:FGSpaces} below, are required to take values in $\sT_G$.  This restriction is mathematically essential, but it is trivial to arrange by whiskering based $G$-spaces to obtain nondegenerately based $G$-spaces, just as in \cite[Appendix]{MayGeo} nonequivariantly.  The following remark explains why we do not just restrict attention to $\sT_G$ everywhere.

\begin{rem}\mylabel{quibble}  A presentation of a pair $(Y,A)$ of $G$-spaces as a $G$-NDR pair induces a presentation of the pair $(\mathbf{Map}(X,Y),\mathbf{Map}(X,A))$ as a $G$-NDR pair when $X$ is a compact $G$-space.  Taking $A=\ast$, it follows that $\mathbf{Map}(X,Y)$ is nondegenerately based by the trivial map  $X\rtarr \ast$ when $Y$ is nondegenerately based.  We deduce that the $G$-space ${\sU_G}_*(X,Y)$ of all based maps between based $G$-spaces is nondegenerately based when $X$ is compact and $Y$ is nondegenerately based.  We do not believe that the conclusion holds for general $X$, a problem overlooked in \cite{MMO}.  Thus we must allow all based spaces and all based $G$-spaces when considering enrichment.
\end{rem}

The categories of unbased $G$-spaces and $G$-maps are enriched over $\sU$, and the categories of based $G$-spaces and based $G$-maps are enriched over $\sU_*$.  
The category $\sU_G$ is enriched over 
$G\sU$, with $G$ acting by conjugation on function spaces, and we may view $G\sU$ as the $G$-fixed point category $(\sU_G)^G$.
Similarly ${\sU_G}_*$ and $\sT_G$ are enriched over $G\sU_*$.  It is important to distinguish between based and unbased enrichment; \cite[\S1.3]{MMO} discusses the importance of this distinction. 

The categories listed above are closed symmetric monoidal with respect to $\times$ for unbased spaces and $\sma$ for based spaces.   It is often best to think of $\sU_G$ and ${\sU_G}_{*}$ not as categories in their own right but rather as giving the hom objects of the closed structures on the symmetric monoidal categories $(G\sU,\times)$ and $(G\sU_{\ast}, \sma)$. This makes sense since for unbased $G$-spaces $X,Y,Z$ we have
$$  G\sU(X\times Y, Z) \iso G\sU(X,\sU_G(Y,Z))$$
and for based $G$-spaces $X,Y,Z$ we have
$$  G\sU_*(X\sma Y, Z) \iso G\sU_*(X,{\sU_G}_{*}(Y,Z)).$$ 

Our definition of $G$-equivalences of $G$-spaces is the usual definition where a $G$-equivalence is detected on fixed points.
 \begin{defn} A $G$-map $f\colon X\rtarr Y$  is a \emph{$G$-equivalence}  if it induces weak homotopy equivalences $f^H\colon X^H\rtarr Y^H$ 
 of fixed point spaces for all subgroups $H\subset G$.
\end{defn}

\subsection{$\sF$-$G$-spaces and $\sF_G$-$G$-spaces}\label{sec:FGSpaces}
We use the following notation.
\begin{itemize}
\item $\sF$ is the subcategory of $\sT$ of based finite  sets $\bf{n}$, where ${\bf n} = \{0,1,\dots, n\}$ with basepoint $0$.
\item $\PI$ is the subcategory of $\sF$ of those maps $\ph\colon \bf m\rtarr \bf n$ such that $|\ph^{-1}(j)| \leq 1$ for $1\leq j\leq n$.
\item $\sF_G$ is the subcategory of $\sT_G$ of based finite $G$-sets of the form $({\bf{n},\al})$, where $\al\colon G\rtarr \SI_n$ is a homomorphism;  
we think of $\al(g)$ for $g\in G$ as giving a based function $g\colon \bf n \rtarr \bf n$.   
\item $\PI_G$ is the subcategory of $\sF_G$ of those maps $\ph\colon (\bf m, \al) \rtarr (\bf n,\be)$ such that $|\ph^{-1}(j)| \leq 1$ for $1\leq j\leq n$.
\end{itemize}

The category $\PI$ was used in \cite{MT} when defining categories of operators. The intuition is that if one thinks of $\sF(\bf m, \bf n)$ as the space of operations with $m$ inputs and $n$ outputs (in particular, $\sF(\bf m , \bf 1)$ is the space of $m$-ary operations), then $\PI$ contains only  those operations that do not combine variables. The category $\PI_G$ was used in \cite{Sant, MMO} to generalize to equivariant categories of operators when indexing on finite $G$-sets.

The category $\sF$ is enriched in based sets and thus, with the discrete topology on morphism sets, in $\sT$. 
We can also view it as enriched in $G\sT$, with the trivial $G$-action on morphism sets. The morphisms of $\sF_G$ 
are the based functions, but now $G$ acts by conjugation on morphism sets and we regard $\sF_G$ as a category enriched in $G\sT$.
The maps in $\PI$ and $\PI_G$ are the composites of injections, projections, and permutations.

\begin{notn}
To make the notation less cumbersome, especially when indices are involved, we  let $\al$ denote  the $G$-set determined by the homomorphism 
$\al\colon G\rtarr \SI_n$, writing $(\bf{n},\al)$ when necessary for clarity.  We sometimes use the notation $|\al|$
to indicate the cardinality $n$ of ${\bf{n}}\backslash\{0\}$. 
\end{notn}

We recall our definitions of $\sF$-$G$-spaces and $\sF_G$-$G$-spaces from \cite[\S 2.1]{MMO}.

\begin{defn}
 An \emph{$\sF$-$G$-space} is a $G\sU_*$-functor $X\colon \sF\rtarr \sT_G$.  Replacing $\sF$ by $\PI$ gives the notion of a \emph{$\PI$-$G$-space}.
\end{defn}
Since $G$ acts trivially on $\sF$, $X$ lands in the fixed point category $(\sT_G)^G=G\sT$.  Thus it is equivalent to define
an $\sF$--$G$-space $X$ to be a $\sU_*$-functor $X\colon \sF\rtarr G\sT$.  

\begin{rem}\mylabel{PIvsF} Since $\PI\subset \sF$, an $\sF$-$G$-space has an underlying $\PI$-$G$-space, but an $\sF$-$G$-space has more structure.  In particular,  it has an $n$-fold multiplication $X_n\rtarr X_1$ coming from the map $\bf n\rtarr \bf 1$ that sends $0$ to $0$ and all other $i$ to $1$. For  $n\geq 2$, this  map is in $\sF$ but not in $\PI$. Due to this extra structure,  the assignment $n\mapsto A^n$ for a $G$-space $A$, which always extends to a functor on $\Pi$, generally does not extend to a functor on
$\sF$.  There is such an extension if and only if $A$ has a product $A^2\rtarr A$ under which it is a  commutative monoid in $G\sT$.

\end{rem}

\begin{notn}
We  write $X_n$ or $X({\bf{n}})$ for the value of $X$ on $\bf{n}$, choosing whichever notation is convenient in context. 
\end{notn}

From the enrichment in $\sU_*$, we see that $X$ is given by based maps 
$$\sF({\bf{m}},{\bf{n}}) \rtarr G\sU_*(X({\bf{m}}),X({\bf{n}}))$$ 
and composition factors through smash products.  Also, as a based functor, any $\sF$-$G$-space $X$ is necessarily reduced, meaning that $X_0=\ast$ (see \cite[Lemma 1.13]{MMO}). There is a unique morphism $\bf{0}\rtarr \bf{n}$ in $\sF$, and the basepoint of $X_n$ is the image of the induced map $X_0\rtarr X_n$.

\begin{defn}
An \emph{$\sF_G$-$G$-space} is a $G\sU_*$-functor $Y\colon \sF_G\rtarr \sT_G$.  Replacing $\sF_G$ with $\PI_G$ gives the notion of a \emph{$\PI_G$-$G$-space}.
\end{defn}
Again, such a $Y$ is necessarily reduced, meaning that $Y(\bf{0})=\ast$. 

\begin{rem}\mylabel{PIGvsFG}
 As in \myref{PIvsF}, an $\sF_G$-$G$-space has an underlying $\PI_G$-$G$-space. The extra structure 
 in $\sF_G$ encodes multiplications on $\sF_G$-$G$-spaces $Y$ that are indexed on finite $G$-sets. We have a multiplication $Y(\alpha)\rtarr Y_1$ where $\al\colon G\rtarr \Sigma_n$  specifies a $G$-action on the set  $\bf n$.
 \end{rem}

We have an adjoint pair of functors given  by restricting along the inclusion $i\colon \sF\rtarr \sF_G$ and prolonging the other way by taking the left Kan extension along $i$.
\begin{equation}\label{ProlongPair} 
\xymatrix@1{ \text{$\sF$-$G$-spaces  } \ar@<0.7ex>[rr]^{\bP_\sF^{\sF_G}}  && \text{ $\sF_G$-$G$-spaces} \ar@<0.7ex>[ll]^{\bU^{\sF_G}_\sF}. }
\end{equation}
We have a similar adjunction using the inclusion $\PI \rtarr \PI_G.$
\begin{equation}\label{ProlongPairPI} 
\xymatrix@1{ \text{$\PI$-$G$-spaces  } \ar@<0.7ex>[rr]^{\bP_\PI^{\PI_G}}  && \text{ $\PI_G$-$G$-spaces} \ar@<0.7ex>[ll]^{\bU^{\PI_G}_\PI}. } 
\end{equation}
The following result originally appeared in \cite{Shim2}  (cf. \cite[Theorem 2.30]{MMO}).  

\begin{prop}\mylabel{ProlongEquiv} The adjoint pairs of \eqref{ProlongPair} and \eqref{ProlongPairPI} are equivalences of categories.
\end{prop}

\subsection{Equivalences and the notion of a Segal machine}
We recall the notions of equivalence for $\sF$-$G$-spaces and $\sF_G$-$G$-spaces from \cite[Definitions 2.6 and 2.27]{MMO}. If $X$ is a $\PI$-$G$-space, the $G$-space $X_n$ has a $\Sigma_n$-action coming from functoriality since permutations are maps in $\PI$. The $G\times \Sigma_n$-action on $X_n$ is used in the following definition of equivalences.

\begin{defn}\mylabel{equivalences} Let $f\colon X\rtarr X'$ be a map of $\PI$-$G$-spaces and let $j\colon Y\rtarr Y'$ be a map of $\PI_G$-$G$-spaces.
\begin{enumerate}[(i)]
 \item The map $f$ is an \emph{\gen-level equivalence} if $f_n\colon X_n^{\LA} \rtarr {X'}_n^{\LA}$ is a weak equivalence for all   $\LA\in \bF_n$, where $\bF_n$ is the set of subgroups $\LA$ of $G\times \SI_n$ that intersect  $\SI_n$ trivially. 
 \item The map $j$ is a \emph{level $G$-equivalence} if $j_{\al}\colon Y(\al)\rtarr Y'(\al)$ is a $G$-equivalence for all finite $G$-sets $\al$.  
 \end{enumerate}
We say that a map of $\sF$-$G$-spaces is an \gen-level equivalence if its underlying map of $\PI$-$G$-spaces is an \gen-level equivalence, and we say that a map of $\sF_G$-$G$-spaces is a level 
$G$-equivalence if its underlying map of $\PI_G$-$G$-spaces is a level $G$-equivalence.
\end{defn}
As explained in \cite[Theorem 2.30]{MMO}, (i)  is what we get if we transport the notion of equivalence of $\sF_G$-$G$-spaces from (ii) along the equivalence of categories \eqref{ProlongPair}. 
Thus we have the following result, which will be used heavily whenever we compare $\sF$-$G$-spaces with $\sF_G$-$G$-spaces.

\begin{thm}\mylabel{equiv}
A map $f$ of $\sF$-$G$-spaces is an \gen-level equivalence if and only if $\bP_{\sF}^{\sF_G} f$ is  a level $G$-equivalence.  
\end{thm}

We also recall what it means for an $\sF$-$G$-space $X$ or an $\sF_G$-$G$-space $Y$ to be special from \cite[Definitions 2.6 and 2.27]{MMO}.  As explained there, there are two notions for $\sF$-$G$-spaces, but we shall only consider the one that gives the correct input for the construction of genuine $G$-spectra.  We first recall the definitions of the Segal maps from \cite[Definitions 2.3 and 2.26]{MMO}, which give more details.  

\begin{defn}\mylabel{RRG} Let $A$ be a based $G$-space. Define  a $\PI$-$G$-space $\bR A$ by letting
its $n$th $G$-space be $A^n =\sT_G(\mathbf{n}, A)$.  Define a $\PI_G$-$G$-space $\bR_G A$ by letting its $\al$th $G$-space be $A^{\al} =   \sT_G(\al,A)$, where $\al$ denotes the based $G$-set $\bf n$ with $G$-action specified by $\al\colon G\rtarr \SI_n$.  

For a $\PI$-$G$-space $X$, define the Segal map $\de\colon X\rtarr \bR X_1$ of $\PI$-$G$-spaces by letting its $n$th $G$-map $\de\colon X_n\rtarr X_1^n$ have coordinates $\de_i$  induced by the $n$ projections $\de_i\colon {\bf n}\rtarr {\bf 1}$, where $\de_i$ sends $i$ to 1 and all other $j$ to $0$. As in \myref{PIvsF}, its target is only a $\PI$-$G$-space even when $X$ is an $\sF$-$G$-space.  

Similarly, for a $\PI_G$-$G$-space $Y$, define the Segal map $\de\colon Y\rtarr \bR_G Y_1$ of 
$\PI_G$-$G$-spaces by letting its $\al$th $G$-map $\de\colon Y(\al) \rtarr \sT_G(\al, Y_1)$ have 
coordinates induced by the projections $\de_i$.  Here, as explained in \cite[Definitions 2.26]{MMO}, 
the $\de_i$ are not $G$-maps, but they nevertheless give the coordinates of the $G$-map $\de$. 
As in \myref{PIGvsFG},  the target is only a $\PI_G$-$G$-space, even when $Y$ is an $\sF_G$-$G$-space.
\end{defn}

\begin{defn}\mylabel{FGspecial} Let $X$ be an $\sF$-$G$-space and $Y$ be an $\sF_G$-$G$-space.
\begin{enumerate}[(i)]
\item $X$ is \emph{\gen-special} if  $\de\colon X\rtarr \bR X_1$ is an \gen-level equivalence of $\PI$-$G$-spaces. 
\item $Y$ is \emph{special} if $\de\colon Y \rtarr \bR_G Y_1$ is a level $G$-equivalence of $\PI_G$-$G$-spaces.
\end{enumerate}
\end{defn}

We note that for $G=e$ an \gen-level equivalence is just a level equivalence, so that for a nonequivariant $\sF$-space we recover the classical definition of 
special as meaning that the maps $X_n\rtarr X_1^n$ are equivalences.  Again by \cite[Theorem 2.30]{MMO}, we have the following companion to \myref{equiv}.

\begin{thm}\mylabel{equivtoo}
An $\sF$-$G$-space $X$ is \gen-special if and only if $\bP_{\sF}^{\sF_G}X$ is special. 
\end{thm}

\begin{defn}\mylabel{Segma} An \emph{(equivariant) Segal machine} is a functor $\shim$ from $\sF$-$G$-spaces $X$ to orthogonal $G$-spectra such that if $X$ is \gen-special,
then 
\begin{enumerate}[(i)]
\item $\shim X$ is a positive $\OM$-$G$-spectrum and
\item if $V^G\neq 0$, there is a natural group completion $X_1\rtarr \OM^V(\shim X)(S^V)$.
\end{enumerate}
\end{defn}

Via the results above, an equivalent definition is that a Segal machine is a functor $\shim$ from $\sF_G$-$G$-spaces $Y$ to orthogonal $G$-spectra such that if $Y$ is special,
then 
\begin{enumerate}[(i)]
\item $\shim Y$ is a positive $\OM$-$G$-spectrum and
\item if $V^G\neq 0$, there is a natural group completion $Y_1\rtarr \OM^V(\shim Y)(S^V)$.
\end{enumerate}
Given a machine $\bS_G$ on $\sF$-$G$-spaces, $\bS_G\com \bU^{\sF_G}_{\sF}$ gives the equivalent machine on $\sF_G$-$G$-spaces.  
Given a machine $\bS_G$ on $\sF_G$-$G$-spaces, $\bS_G\com \bP^{\sF_G}_{\sF}$ gives the equivalent machine on $\sF_G$-$G$-spaces. 

The original non-equivariant Segal machine was introduced by Segal in \cite{Seg}. An equivariant version of the Segal machine was constructed in  \cite{Shim, MMO}. 
Our goal in this paper is to construct equivalent machines with better properties.

\section{The invariance theorem}\label{inv}
All variants of the Segal machine treated here start with $\sF$-$G$-spaces or $\sF_G$-$G$-spaces,
prolong them to functors defined on based $G$-CW complexes $A$, and then restrict $A$ to spheres $S^V$ to obtain orthogonal $G$-spectra. 
To compare machines, we show that prolongation preserves equivalences under suitable hypotheses.  

\subsection{Statement of the invariance theorem}  
We recall the details of the prolongation functor from \cite[\S2.3]{MMO} and state the invariance theorem.

\begin{defn}\label{DefWGSpace}
Let $\sW_G$ be the full subcategory of $\sT_G$ of based $G$-CW complexes, and define a \emph{$\sW_G$-$G$-space}
to be a $G\sU_*$-functor $X\colon \sW_G\rtarr \sT_G$.
\end{defn}

\begin{con}\mylabel{wgspaces}
The $G$-spaces $X(S^V)$ of a $\sW_G$-$G$-space specify an orthogonal $G$-spectrum. 
Identifying $S^V\sma S^W$ with $S^{V\oplus W}$,  its structure $G$-maps 
$$X(S^V)\sma S^W \rtarr X(S^{V\oplus W})$$  
are adjoint to the composites
\[ \xymatrix@1{S^W \ar[r] & \sT_G(S^V,S^{V\oplus W}) \ar[r]^-{X} & \sT_G(X(S^V),X(S^{V\oplus W})),} \]
where the first map is adjoint to the identity map of $S^{V\oplus W}$.
\end{con}

We define the prolongation functor $\bP_{\sF}^{\sW_G}$ from $\sF$-$G$-spaces to $\sW_G$-$G$-spaces to be the left adjoint of the functor induced by the inclusion $\sF \rtarr \sW_G$. For a $G$-CW-complex $A$, its evaluation at $A$ is given by the categorical tensor product
\[ (\bP_{\sF}^{\sW_G} X)(A) = A^\bullet\otimes_\sF X,\]
where $A^\bullet$ is the functor $\sF^{op}\rtarr \sT_G$ that sends $\mathbf{n}$ to the based $G$-space  $A^n = \sT_G(\mathbf{n},A).$  A detailed point-set topological level discussion of the prolongation functor is given in \cite[Appendix B]{Schwede2}. We have a factorization of prolongation functors as in the following diagram.  This factorization (up to canonical natural isomorphism) of left adjoints is immediate since the 
corresponding factorization of their right adjoint forgetful functors is evident.

\[ \xymatrix @R=2ex{
 & \text{$\sF_G$-$G$-spaces} \ar[dr]^{\bP_{\sF_G}^{\sW_G}}  & \\
\text{$\sF$-$G$-spaces}  \ar[ur]^{\bP_{\sF}^{\sF_G}} \ar@/_0ex/[rr]_{\bP_{\sF}^{\sW_G}} & &
\text{$\sW_G$-$G$-spaces}
}\]

As recalled in \myref{ProlongEquiv}, the functor $\bP_{\sF}^{\sF_G}$ on the left is an equivalence of categories.
However, the prolongation functor $\bP_{\sF_G}^{\sW_G}$ does not preserve level $G$-equivalences in general, hence we cannot expect the composite functor $\bP_{\sF}^{\sW_G}$ to take \gen-level equivalences of $\sF$-$G$-spaces to level $G$-equivalences of $\sW_G$-$G$-spaces in general. We will prove that such a homotopical invariance theorem does hold if we  impose the following ``cofibrancy" condition on $\sF$-$G$-spaces.  A simplicial $G$-space is a simplicial object in $G$-spaces; equivalently, it is a  $G$-object in simplicial spaces. It was shown in \cite[\S1.2]{MMO}) that a simplicial $G$-space $X_\sbt$ is Reedy cofibrant if all of its degeneracy maps $s_i \colon X_n \rtarr X_{n+1}$ are $G$-cofibrations.

\begin{defn}\mylabel{cof}
An $\sF$-$G$-space $X$ is \emph{proper} if for any based simplicial $G$-set $A_\sbt$ the simplicial $G$-space $(\bP_{\sF}^{\sW_G} X)(A_\sbt)$ is Reedy cofibrant.  An $\sF_G$-$G$-space $Y$ is \emph{proper} if for any simplicial $G$-set $A_\sbt$ the simplicial $G$-space $(\bP_{\sF_G}^{\sW_G} Y)(A_\sbt)$ is Reedy cofibrant.  Observe that $X$ is proper if and only if $\bP_{\sF}^{\sF_G} X$ is proper. 
\end{defn} 

This notion of a proper $\sF$-$G$-space (or $\sF_G$-$G$-space)  is new and may look strange at first sight.\footnote{We owe the idea for this notion to Stefan Schwede.}  Recall that all simplicial sets are Reedy cofibrant in the standard Quillen model structure.  
Since cofibrations are precisely monomorphisms, the same is true for 
simplicial $G$-sets with the model structure given by requiring a map $f\colon K\rtarr L$ of simplicial $G$-sets to be a weak equivalence or fibration if each fixed point map $f^H$ is a weak equivalence or fibration for all subgroups $H\subset G$; see \cite[Proposition~2.16]{Steph}.

\begin{rem}\mylabel{Reedy}
In \cite[Definition 9.10]{MMO} we said that a $\sW_G$-$G$-space $Z$, such as $\bP_{\sF}^{\sW_G} X$ or $\bP_{\sF_G}^{\sW_G} Y$, preserves Reedy cofibrancy if for every 
simplicial $G$-CW-complex $A_\sbt$, the simplicial $G$-space $Z(A_\sbt)$ is Reedy cofibrant.   Clearly $X$ is proper if $\bP_{\sF}^{\sW_G} X$ preserves Reedy cofibrancy.
\end{rem}

We can now state the invariance theorem.

\begin{thm}[Invariance theorem]\mylabel{inv1}
Let $f\colon X\rtarr X'$ be an \gen-level equivalence of proper $\sF$-$G$-spaces. Then the induced map 
$$\bP_{\sF}^{\sW_G} f\colon (\bP_{\sF}^{\sW_G} X)(A)  \rtarr (\bP_{\sF}^{\sW_G} X')(A)$$
is a $G$-equivalence for all based $G$-CW complexes $A$. 
\end{thm}

By \myref{equiv} and the observation at the end of \myref{cof}, the invariance theorem admits the following equivalent reinterpretation.

\begin{thm}[Invariance theorem]\mylabel{inv2}
Let $f\colon Y\rtarr Y'$ be a level equivalence of proper $\sF_G$-$G$-spaces. Then the induced map 
$$ \bP_{\sF_G}^{\sW_G} f\colon (\bP_{\sF_G}^{\sW_G}Y)(A)  \rtarr (\bP_{\sF_G}^{\sW_G}Y')(A)$$
is a $G$-equivalence for all based $G$-CW complexes $A$. 
\end{thm}

\begin{rem}
Note that when $G=e$, an \gen-level equivalence of $\sF$-spaces is just a level equivalence. Thus, in the nonequivariant case, the invariance theorem says that prolongation preserves level equivalences between proper $\sF$-spaces.
\end{rem}

\subsection{Proof of the invariance theorem}
We make use of the classical adjunction $(|-|, S_\sbt)$ between geometric realization of simplicial sets and the total singular complex functor. 
Let $G$ act trivially on the topological simplices $\DE^n$.  If $A$ is a based $G$-space, then $S_\sbt A$ is a based $G$-simplicial set with $G$ acting on simplices via the action of $G$ on $A$.   Visibly $(S_\sbt A)^H = S_\sbt (A^H)$; a simplex $f\colon \DE^n \rtarr A$ is fixed by $H$  if and only it takes 
values in $A^H$.  Similarly, for a simplicial $G$-set $K_\sbt$,  we have $(|K_\sbt |)^H = |K_\sbt^H|$.  Thus the natural $G$-map
$\epz\colon |S_\sbt A| \rtarr A$ restricts on $H$-fixed points to the standard weak equivalence $|S_\sbt A^H|\rtarr A^H$.  Moreover, just as nonequivariantly, $|K_\sbt |$ is a $G$-CW complex.  Choosing one $H$ in each conjugacy class of subgroups of $G$, it has one cell of type $G/H \times D^n$ for each nondegenerate $n$-simplex with isotropy group $H$.  If $A$ is a $G$-CW complex,
$\epz$ is therefore a $G$-homotopy equivalence.   This implies the following lemma.

\begin{lem} For any $\sF$-$G$-space $X$, the map $\epz\colon |S_\sbt A| \rtarr A$ induces a natural $G$-homotopy equivalence 
\[\bP_{\sF}^{\sW_G}X( \epz) \colon |S_\sbt A|^\bullet\otimes_\sF X \to A^\bullet \otimes_\sF X.\]
\end{lem}

Therefore the invariance theorem holds if and only if its conclusion holds with $A$ replaced by $|S_\sbt A|$.   We may view
a $G$-set $A$, not necessarily finite, as a $G$-CW complex.  The following lemma gives the starting point for the proof
of the invariance theorem. 

\begin{lem}\mylabel{InvarGSet}  Let $f\colon X\rtarr X'$ be an \gen-level equivalence of $\sF$-$G$-spaces. Then 
$\bP_{\sF}^{\sW_G} f\colon (\bP_{\sF}^{\sW_G} X)(A) \rtarr (\bP_{\sF}^{\sW_G} X')(A)$ is a $G$-equivalence for any based $G$-set $A$. 
\end{lem}
\begin{proof}  By Zorn's Lemma, we may decompose the based $G$-set $A$ as a wedge of orbits $G/H_+$.
We can then order our orbits, so that $A$ is the well-ordered colimit of maps of the form $B_+ \rtarr B_+\vee G/H_+$.
Due to the presence of the disjoint basepoint, the inclusion from one term into the next is the inclusion of a retract, a property 
that is retained on application of $\bP_{\sF}^{\sW_G}X$.   
By \cite[Lemma~1.6.2]{MaySig}, it follows that each induced map $(\bP_{\sF}^{\sW_G} X)(B_+) \rtarr (\bP_{\sF}^{\sW_G} X)(B_+\vee G/H_+)$ is a closed inclusion.

Since passage to $H$-fixed points commutes with wedges and ordered colimits (\cite[Lemma~III.1.6]{MM}) and since we can commute ordered colimits and the coequalizers that define $\bP_{\sF}^{\sW_G}$, the map $\bP_{\sF}^{\sW_G} f$, evaluated at $A$, is the colimit of an ordered set of $G$-equivalences. It is therefore a $G$-equivalence since passage to homotopy groups commutes with ordered colimits of closed inclusions.  In more detail, we again use that the colimits in question commute with fixed points and we observe that applying fixed points to a closed inclusion again produces a closed inclusion. The fact that homotopy groups commute with colimits of sequences in the category $\mathbf{Top}$  of (arbitrary) topological spaces is classical \cite[Lemma 2.14]{DT}, and the classical proof readily generalizes from sequences to ordered colimits.  As all maps in the colimit system are closed inclusions, the colimit as calculated in $\mathbf{Top}$ agrees with the colimit as calculated in $\sU$ \cite[Lemma~3.3]{CGWH}.
\end{proof}

Now we use the notion of a proper $\sF$-$G$-space to complete the proof of the invariance theorem.

\begin{prop}\mylabel{InvarProp}
Let $f\colon X\to X'$ be an \gen-level equivalence of proper $\sF$-$G$-spaces and let $A=|K_{\sbt}|$ be the geometric realization 
of a   based $G$-simplicial set $K_{\sbt}$. Then $\bP_{\sF}^{\sW_G} f\colon (\bP_{\sF}^{\sW_G} X)(A)\rtarr (\bP_{\sF}^{\sW_G} X')(A) $ is a $G$-equivalence. 
\end{prop}
\begin{proof}  Using  that products commute with realization and that 
the coequalizers that define $\bP_{\sF}^{\sW_G}$ commute
with the colimits that define $|-|$, we see that 
\[ A^\bullet\otimes_\sF X \iso  | (\bP_{\sF}^{\sW_G} X)(K_\sbt) |,\] and similarly for $X'$.  By \myref{InvarGSet}, $(\bP_{\sF}^{\sW_G} f)(A)$ is the realization of a level $G$-equivalence of simplicial $G$-spaces, and by assumption $(\bP_{\sF}^{\sW_G} X)(K_\sbt)$ and $(\bP_{\sF}^{\sW_G} X')(K_\sbt)$ are Reedy cofibrant simplicial $G$-spaces. Therefore,  $(\bP_{\sF}^{\sW_G} f)(A)$  is a $G$-equivalence \cite[Theorem~1.10]{MMO}. 
\end{proof}

\begin{rem}  We only need the case $A=S^V$.  As a smooth $G$-manifold,
$S^V$ is triangulable as a countable $G$-simplicial complex \cite{Illman}.
Restricting to this case, we only need countable colimits in a variant of the argument above.  Schwede \cite[Proposition B.48]{Schwede2} gives an analogous invariance theorem that restricts attention 
to finite $G$-CW complexes $A$ and bypasses the use of infinite colimits.
\end{rem}

\section{Variants of the two-sided bar construction}\label{variants}

All of our variants of the Segal machine are constructed by prolonging either  $\sF$-$G$-spaces or $\sF_G$-$G$-spaces given by two-sided bar constructions  to $\sW_G$-$G$-spaces and then restricting to $G$-spheres to obtain (orthogonal) $G$-spectra.  We focus on the relevant bar constructions in this section.   We first recall the general definitions and then specialize to the examples of interest.

\subsection{The general monadic bar construction $B(\bY,\bE,X)$}\label{formal}

We assume given a closed symmetric monoidal category $\sV$ with product $\otimes$ and internal hom objects $\ul{\sV}(V,W)$.
 We are thinking of $(\sV,\otimes)$ as either $(G\sU,\times)$ or $(G\sU_*,\sma)$, and then we are thinking of $\sU_G$ 
and ${\sU_G}_{*}$ as giving the hom objects of the closed structures on $(G\sU,\times)$ and $(G\sU_{\ast}, \sma)$.  We also assume 
given a small $\sV$-category $\sE$.  We are thinking of $\sE$ as either $\sF$ or $\sF_G$.  With these examples in mind, we shall be a little imprecise about categorical details in order to avoid excessive pedantry.    

In decades of previous work, and especially in the prequel \cite{MMO}, it is emphasized that there is a categorical two-sided bar construction that is defined when $\otimes = \times$ (as in \cite [\S12]{MayClass}) and a monadic two-sided bar construction that is defined in general (as in \cite[\S9]{MayGeo}).  It is essential to our new examples that we work monadically.  In this section, we recall the definition of  the monadic bar construction and show that the categorical bar construction is a special case. We then generalize this special case to obtain the monadic bar constructions of interest in this paper.   In our examples, $\sW$ in the following definition will be the category $\Fun (\PS,\sV)$ of $\sV$-functors $\PS\rtarr \sV$ for some $\sV$-subcategory $\PS$ of $\sE$.\footnote{More generally, $\sV$ could be replaced by an appropriate $\sV$-category $\sM$.}
We require the category  $\sZ$ below to be enriched and tensored over $\sV$.
 
\begin{defn}
Let $(\bE,\mu,\eta)$ be a monad in a ground category $\sW$, let $(X,\theta)$ be an $\bE$-algebra, and let
$\bY\colon \sW\rtarr \sZ$ be a right $\bE$-functor, namely a functor together with an action natural transformation
$\vartheta \colon \bY\bE \rtarr \bY$ such that $\vartheta\com \bY\eta = \bI$ and $\vartheta \com \bY\mu = \vartheta\com \vartheta \bE$. Then  the bar construction $B(\bY,\bE,X)$ in $\sZ$ is defined as the geometric realization, constructed as usual in our topological examples, of a simplicial object $B_\sbt(\bY,\bE,X)$ in $\sZ$  with $q$-simplices  $\bY \bE^{q} X$.  The faces are induced by $\vartheta$, $\mu$, and the action $\theta\colon \bE X \rtarr X$ and the degeneracies are induced by $\et$, as in \cite[\S9]{MayGeo}. 
\end{defn}

Conceptually, the bar construction  $B(\bY,\bE,X)$ is a derived variant of the monadic tensor product  $\bY\otimes_{\bE}X$, which is defined to be the coequalizer of the pair of maps $$\bY\bE X \rightrightarrows \bY X$$ in $\sZ$ induced by $\vartheta$ and $\tha$.  The maps $\bY \bE^{q} X \rtarr \bY X$ given by iterated use of $\mu$ induce a map of simplicial objects in $\sZ$ from $B_\sbt(\bY,\bE,X)$ to the constant simplicial object at $\bY\otimes_{\bE} X$.  Its realization is a natural map
$$\epz\colon B(\bY,\bE,X)\rtarr \bY\otimes_{\bE}X.$$

The most obvious example of $\bY$ is $\bE\colon \sW\rtarr \sW$, with $\vartheta = \mu$.  In this case $\mu$ also induces an
$\bE$-algebra structure on $B(\bE,\bE,X)$, and $\epz\colon B(\bE,\bE,X) \rtarr X$ is a map 
of $\bE$-algebras.  By the standard extra degeneracy argument, as in \cite{MayGeo, Shul}, $\epz$ is a homotopy equivalence when $\sW$ is any of the categories of spaces or $G$-spaces we consider, with homotopy inverse $\et\colon X \rtarr B(\bE,\bE,X)$ induced by $\et\colon X \rtarr \bE X$ on $0$-simplices. Examples of this form are the starting point of all of our variants of the Segal machine. By inspection in examples or a formal argument in general \cite{Shul},
we have a natural isomorphism
\begin{equation}\mylabel{BAotimes}
\bY\otimes_{\bE} B(\bE,\bE,X)\iso B(\bY,\bE,X)
\end{equation}
over $\bY\otimes_{\bE}X$.  

Recall that any adjoint pair of functors $(\bP,\bU)$ gives rise to a monad $\bE = \bU\bP$  with unit $\Id \rtarr \bE$ given by the unit $\et$ of the adjunction, and product $\bE\bE\rtarr \bE$ given by $\bU \epz \bP$, where $\epz$ is the counit of the adjunction.  Here the $\sW$ above is the domain category of $\bP$. Our monads $\bE$ are all of this form.

To compare with the categorical bar construction (in its general enriched form), we let
$\sO$ denote the object set of $\sE$.  We think of $\sO$ as a discrete $\sV$-category, identity morphisms only, and we have the inclusion of $\sV$-categories $\sO \rtarr \sE$. Notice that a covariant or contravariant $\sV$-functor $X\colon \sO \rtarr \sV$ is just an assigment of objects $X(e)$ of 
$\sV$, one for each object $e\in \sO$ of $\sE$. We then have a forgetful functor 
\begin{equation}
 \bU = \bU^{\sE}_{\sO} \colon {\Fun }(\sE,\sV) \rtarr {\Fun }(\sO,\sV). 
 \end{equation}
Assuming (as we have already done implicitly) that $\sV$ has coproducts, $\bU$ has the left adjoint 
 \begin{equation}
  \bP = \bP^{\sE}_{\sO} \colon {\Fun }(\sO,\sV) \rtarr {\Fun }(\sE,\sV)
 \end{equation}
 specified by
 \begin{equation}
 (\bP X)(e) = \coprod_{d}  \sE(d,e) \otimes X(d).
 \end{equation}
 The evaluation map of this $\sV$-functor is given by coproducts of maps
 $$ \sE(e,f)\otimes \sE(d,e) \otimes X(d) \rtarr \sE(d,f) \otimes X(d) $$
 given by composition in $\sE$.   Remembering only the underlying objects, this gives the monad $\bE = \bU\bP$ in 
the ground category $\sW = {\Fun }(\sO,\sV)$.
 
In this case, an action $\tha\colon \bE X \rtarr X$ is just the evaluation map of a covariant 
$\sV$-functor $X\colon \sE\rtarr \sV$.   
For a contravariant $\sV$-functor $Y\colon \sE\rtarr \sZ$, where $\sZ$ is a $\sV$-category tensored over  $\sV$, define a right $\bE$-functor 
 $\bY\colon {\Fun }(\sO,\sV)\rtarr\sZ$ by
$$ \bY(Z) = \coprod_e Y(e) \otimes Z(e).$$
Then $(\bY \bE)(Z) = \coprod_{e,d} Y(e) \otimes \sE(d,e) \otimes Z(d)$, and $\vartheta$ is induced by the evaluation map 
$Y(e) \otimes \sE(d,e)\rtarr Y(d)$ of the $\sV$-functor $Y$.  Expanding the definitions, the monadic bar construction is given by
\begin{equation}
B_q(\bY,\bE,X) = \coprod_{(e_0,\dots, e_q)} Y(e_q)\otimes \sE(e_{q-1},e_q)\otimes \cdots \otimes \sE(e_0,e_1) \otimes X(e_0).
\end{equation}
This is the same as the  $B_q(Y,\sE,X)$ of the categorical bar construction, and the faces and degeneracies agree.  
Therefore $B(\bY,\bE,X)$ coincides with the categorical two-sided bar construction $B(Y,\sE,X)$.   

An object $A\in \sV$ gives the represented contravariant $\sV$-functor $Y = \ul{\sV}(-,A)$ on $\sV$. If $\sE \subseteq \sV$, 
the associated right $\bE$-functor $\bY\colon \Fun(\sO,\sV) \to \sV$ is denoted by $A^{\bullet}$.
It is defined explicitly on $\sV$-functors 
$Z\colon \sO \rtarr \sV$ by
\[  A^{\bullet}(Z) = \coprod_{e} \sV(e, A)\otimes Z(e). \]
These right $\bE$-functors lead to the prolongation functors that are used to construct the Segal machine.

The treatment of the Segal machine in \cite{Shim, MMO} uses the categorical bar construction with $\sV = G\sU$.  We have analogous monads and monadic bar constructions starting from $\sV = G\sU_*$.  We did not use those in \cite{MMO} only because we did not yet have the invariance theorem and so did not know they behave well homotopically.   That gives one innovation in this paper.  But the main innovation is to replace $\sO$  by  the subgroupoid all isomorphisms of $\sE$.  Anticipating our examples, we denote this subgroupoid by $\SI$.  We again have a forgetful functor
\begin{equation}
 \bU = \bU^{\sE}_{\SI} \colon {\Fun }(\sE,\sV) \rtarr {\Fun }(\SI,\sV). 
 \end{equation}
 It has the left adjoint 
 \begin{equation}
  \bP = \bP^{\sE}_{\SI} \colon {\Fun }(\SI,\sV) \rtarr {\Fun }(\sE,\sV),
 \end{equation}
given by  the categorical tensor product of functors $\sE\otimes_{\SI} X$.  Thus, for an object $e$ of $\sE$,  $\bP(X)(e)$ is the coequalizer of the pair of maps
\[ \coprod_{(c,d)} \sE(d,e) \otimes \SI(c,d) \otimes X(c) \rightrightarrows \coprod_{d} \sE(d,e)\otimes X(d) \]
given by $\com\otimes \id$ and $\id\otimes \ze$, where $\com\colon \sE(d,e) \otimes \SI(c,d)\rtarr \sE(c,e)$ is composition 
in $\sE$ and $\ze\colon \SI(c,d) \otimes X(c) \rtarr X(d)$ is the evaluation map of the $\sV$-functor $X\colon \SI\rtarr \sV$. 
This gives a monad $\bE = \bU\bP$ in the ground category $\sW ={\Fun }(\SI,\sV)$.

Here again, for an object $A\in \sV$, we have a right $\bE$-functor $A^{\bullet}$.  It is defined on $\sV$-functors 
$Z\colon \SI\rtarr \sV$ as the categorical tensor product of the functor represented by $A$ and $Z$.  
Explicitly, it is the coequalizer  of the pair of maps
\[ \coprod_{(c,d)}  \sV(d, A) \otimes \SI(c,d) \otimes Z(c) \rightrightarrows \coprod_{d} \sV(d,A)\otimes Z(d) \]
given by composition in $\sV$ and the evaluation map of $Z$.  Assuming that $\SI\subset \sE\subset \sV$, 
the right action of $\bE$ on $A^{\bullet}$ is induced by composition in $\sV$.  To see this, it helps to observe 
that, ignoring associativity isomorphisms, $A^{\bullet}\bE Z$ is constructed by passage to coequalizers from
$$ \coprod_{a,b,c,d} \sV(d,A)\otimes \SI(c,d) \otimes \sE(b,c) \otimes \SI(a,b) \otimes Z(a).$$

\subsection{Overview of the relevant specializations}\label{concrete} 

In principle, there are eight relevant specializations visible in our general discussion.  Four are obtained by taking $\sE=\sF$
and  four are obtained by taking $\sE = \sF_G$.  In each case, two take $(\sV,\otimes) =(G\sU,\times)$ and two take 
$(\sV,\otimes) =(G\sU_*,\sma)$ for the enriching category, while two start from $\sO\subset \sE$ and two start from $\SI\subset \sE$.  

We first consider the distinction between $G\sU$ and $G\sU_*$.   Thus consider the functor categories
\begin{equation}\label{Fun}
 {\Fun }(\sE,\sU_G) \ \ \text{and} \ \   {\Fun }(\sE,{\sU_G}_*). 
\end{equation}
We require the first to consist of $G\sU$-functors and the second to consist of $G\sU_*$-functors; that is, the functors $X$ in these categories must be given by maps
\[   X\colon \sE(d,e) \rtarr \sU_G(X(d),X(e))    \ \ \ \text{or}  \ \ \ X\colon \sE(d,e) \rtarr {\sU_G}_*(X(d),X(e)) \]
in $G\sU$ or $G\sU_*$.  As said earlier, in the case of ${\sU_G}_*$ this forces $X$ to be reduced in the sense  that $X({\mathbf{0}})$ is a point.  In particular, $\sF$-$G$-spaces and $\sF_G$-$G$-spaces are reduced.  
Thus, the full subcategory of ${\Fun}(\sE,\sU_G)$ consisting of those functors that are reduced can be identified with ${\Fun}(\sE,{\sU_G}_*).$  Indeed,  
applying $X$ to the unique map $\mathbf{0}\rtarr \mathbf{n}$ in $\sF$ then gives the $G$-spaces $X(\mathbf{n})$ basepoints that are preserved by $X(\ph)$ for all morphisms $\ph\in \sF$, and similarly for $\sF_G$.

However, for any of the four $\bE$ that we have in the case of $(G\sU,\times)$, the bar construction $B(\bE,\bE,X)$
 is not reduced even when $X$ is so.  To remedy that and to get enrichment in 
$G\sU_*$ rather than in $G\sU$, we follow \cite[(3.3)]{MMO} and replace all bar constructions $B(\bY,\bE,X)$ with the reduced variants obtained by quotienting out the contractible $G$-subspace $B(\ast, \bE,X)$.  Here we use that the basepoints of the $Y_n$ induced by the maps $\mathbf{0} \rtarr \mathbf{n}$ in $\sF$ or  $\sF_G$ give a map
$\ast\rtarr \bY$ of $\bE$-functors for any $\bE$-functor $\bY$. In particular, we define
\begin{equation}
\widetilde{B}(\bE, \bE,X) = B(\bE, \bE,X)/B(\ast, \bE,X).
\end{equation}

In \cite{MMO}, we used the unadorned notation $B$ for such reduced bar constructions.  
Taking $\sE = \sF_G$ and restricting our target to be  $\sT_G$ rather than ${\sU_G}_*$,\footnote{By inspection of definitions, all of the monads in sight restrict to $\sT_G$, meaning that if $X$ takes values in $\sT_G$, then so do $\bE X$ and the relevant bar constructions.} the reduced bar construction gives rise to the versions of the Segal machines used there.  To describe our variants, we fix the following notations.

\begin{notn} 
The set of objects of $\sF$ is the set $\bN$ of natural numbers, so we write $\sO = \bN$ in that case.  
Correspondingly, we write $\sO =\bN_G$ for the set of objects $({\mathbf n},\al)$ of $\sF_G$.  
The maximal subgroupoid of $\sF$ is the disjoint union $\SI$ of the symmetric groups $\SI_n$,
and here the source and target of an isomorphism must be the same.   Correspondingly we write
$\SI_G$ for the maximal subgroupoid of $\sF_G$.  Thus $\SI_G$ has the same objects as $\sF_G$ and, for finite based $G$-sets 
$\al$ and $\be$, $\SI_G(\al,\be)$ is empty if $|\al|\neq |\be|$, and it is $\SI_n$, equipped with the conjugation 
$G$-action, if $|\al|=|\be|=n$.
\end{notn}

We have the following eight monads, listed here with the corresponding ground categories
\begin{equation}\label{monadF}
\begin{aligned}
&^{\times}\bF^{\bN} =({\Fun}(\bN,\sU_G),  \bU^{\sF}_{\bN} \bP^{\sF}_{\bN}) &  & ^{\times}\bF_G^{\bN_G}=({\Fun}(\bN_G,\sU_G),  \bU^{\sF_G}_{\bN_G} \bP^{\sF_G}_{\bN_G}) \\  
&^{\times}\bF^{\SI} =({\Fun}(\SI,\sU_G),  \bU^{\sF}_{\SI} \bP^{\sF}_{\SI})  &    &  ^{\times}\bF_G^{\SI_G} =({\Fun}(\SI_G,\sU_G),  \bU^{\sF_G}_{\SI_G} \bP^{\sF_G}_{\SI_G} )\\
&^{\sma}\bF^{\bN} =({\Fun}(\bN,\sT_G),  \bU^{\sF}_{\bN} \bP^{\sF}_{\bN})  & & ^{\sma}\bF_G^{\bN_G}=  ({\Fun}(\bN_G,\sT_G),\bU^{\sF_G}_{\bN_G} \bP^{\sF_G}_{\bN_G} ) \\   
&^{\sma}\bF^{\SI} =({\Fun}(\SI,\sT_G), \bU^{\sF}_{\SI} \bP^{\sF}_{\SI})  &   &  ^{\sma}\bF_G^{\SI_G}  =({\Fun}(\SI_G,\sT_G), \bU^{\sF_G}_{\SI_G} \bP^{\sF_G}_{\SI_G} )
\end{aligned}
\end{equation}

\begin{notn}\mylabel{simp}  By default, we agree  to write $\bF$ rather than $^{\sma}\bF$ for the bottom four monads, occasionally retaining the notation $^{\sma}\bF$ for emphasis or when making comparisons. 
We always write $^{\times}\bF$  when the $\times$-variant is meant.  
\end{notn}

The algebras for the four monads on the left are functors with domain $\sF$.  The two on the top are unreduced functors, and the two on the bottom are reduced functors. Similarly, the algebras for the four monads on the right are functors with domain $\sF_G$, with the two on the top being unreduced and the two on the bottom being reduced. For the four categories of algebras for the monads on the top, we will restrict our attention to those algebras whose underlying functor is reduced (and hence based), and lands in $\sT_G$. Recalling the equivalence between the categories of $\sF$-$G$-categories and $\sF_G$-$G$-categories, we see that we have eight monads in sight, all of which have equivalent categories of algebras, after the appropriate restrictions are taken.  There are still others that we have chosen not to consider; see \myref{PIPIG}.

We have eight corresponding bar constructions.  Fixing an $\sF$-$G$-space $X$ and an 
$\sF_G$-$G$-space $Y$,
we obtain eight $\sW_G$-$G$-spaces by taking $A^{\bullet}$ for $A\in \sW_G$ as the first variable in the bar constructions.   As a reminder that we must reduce the bar constructions in the $\times$-case  and that we prefer the $\sma$-case, we systematically write $\widetilde{B}$ in the $\times$-case and $B$ in the $\sma$-case.   

For example, we write
\begin{equation}\label{drei}
 (\widetilde{B}^{\bN_G} Y)(A) = \widetilde{B}(A^\bullet, ^{\times}\bF_G^{\bN_G},Y) \ \ \  \text{and}\ \ \  (B^{\SI} X)(A) = B(A^\bullet, ^{\sma}\bF^{\SI},X).
\end{equation}

We display these  $\sW_G$-$G$-spaces in a commutative diagram.  For the isomorphism $r$, which is given by  \myref{MAIN1}, we take $Y=\bP^{\sF_G}_{\sF} X$ or, equivalently,  $X = \bU^{\sF_G}_{\sF} Y$.

The dotted arrow maps $p$ are equivalences, but the proof is not obvious and will be omitted.  The left dotted arrow $q$ is thus also an equivalence since by \myref{MAIN2} the bottom left solid arrow $q$ is an equivalence.  The right dotted arrow $q$ is not an equivalence since the bottom right solid arrow $q$ is not an equivalence, as we shall see. However, the dotted arrows are included solely in order to make the relationships among the eight bar constructions clear since we have no real interest in the symmetric bar constructions defined using $\times$ rather than $\sma$.

\begin{equation}\label{ein}
\xymatrix{
\widetilde{B}^{\bN_G}Y \ar@{-->}[r]^-{q} \ar[d]_{p} & \widetilde{B}^{\SI_G}Y \ar@{-->}[d]^{p}  & \widetilde{B}^{\SI} X  
\ar@{-->}[d]^{p}  & \widetilde{B}^\bN X\ar[d]^{p} \ar@{-->}[l]_-{q} \\
B^{\bN_G}Y  \ar[r]_-{q}  &  \  B^{\SI_G} Y     \ar[r]^-{\iso}_-{r} & B^{\SI}X &  \ar[l]^-{q}  B^\bN X \\}
\end{equation}

The natural quotient map from products to smash products of $G$-spaces induces quotient maps labeled $p$.  The natural quotient maps from monads defined just using objects to monads defined using isomorphisms of objects induce quotient maps labeled $q$.

It was convenient to use $\times$ for some of the proofs in \cite{MMO}. However, as the use of dotted arrows emphasizes, we see no advantages to using $\times$ equivariantly or multiplicatively, preferring to use $\sma$ on the grounds that it is most natural to work throughout with based $G$-spaces.
The invariance theorem allows us to do so since it allows us to prove that the solid arrow maps $p$ 
in (\ref{ein}) are level $G$-equivalence of $\sW_G$-$G$-spaces; see \myref{smatimes1}. 

The diagram gives eight candidates for Segal machines. According with our preference for the variant of the bar construction defined using $\sma$, we rule out $\widetilde{B}^{\SI_G}$ and $\widetilde{B}^{\SI}$ from further consideration. Equivariantly, as we explain in \S\ref{NNG}, $\widetilde{B}^\bN$ and $B^\bN${\em fail} to give Segal machines, and we therefore also rule them out.  As we shall prove in the following sections, we are left with four equivalent Segal machines, two of which are symmetric monoidal and two or which are not. 

Nonequivariantly, with $G=e$, the left and right squares coincide and we have agreed to ignore $\widetilde{B}^{\SI}$.  Here $\widetilde{B}^\bN$ gives Woolfson's variant \cite{Woolf} of Segal's original machine \cite{Seg} and our $B^{\bN}$ gives an equivalent variant defined using $\sma$ instead of $\times$.   These machines are {\em not} symmetric monoidal, but our $B^{\SI}$ gives a third equivalent nonequivariant machine that is so. 

Equivariantly, as we explain in  \S\ref{NNG}, $\widetilde{B}^{\bN_G}$ gives the Segal machine studied in \cite{Shim, MMO}, and $B^{\bN_G}$ gives an equivalent machine defined using $\sma$ instead of 
$\times$.   Again, these machines are \emph{not} symmetric monoidal.

The main new idea of this paper is the introduction of the symmetric bar constructions defined using $\SI$ and $\SI_G$.  
In \myref{MAIN1}, we shall prove the surprising fact that $B^{\SI}$ and $B^{\SI_G}$ give isomorphic $\sW_G$-$G$-spaces, giving the isomorphism $r$ in (\ref{ein}).
In \myref{MAIN2}, we shall prove that the bottom left arrow $q\colon B^{\bN_G}Y\rtarr  B^{\SI_G} Y  $
 in (\ref{ein}) is a level $G$-equivalence of $\sW_G$-$G$-spaces.  These results prove the main result of the paper, namely that the composite 
 \begin{equation}\label{rqp}
 \xymatrix@1{
 \widetilde{B}^{\bN_G}Y \ar[r]^-{p} & B^{\bN_G}Y  \ar[r]^-{q}  &  \  B^{\SI_G} Y     \ar[r]_-{\iso}^-{r} & B^{\SI}X\\} 
 \end{equation}
displays equivalences between four Segal machines when $Y=\bP^{\sF_G}_{\sF} X$ or, equivalently,  
when $X = \bU^{\sF_G}_{\sF} Y$.  The 
isomorphic symmetric variants remedy the defects of the original Segal machine of \cite{Shim, MMO}. 
 
\subsection{The bar constructions defined using $\bN_G$}\label{NNG}

We start with $\sF_G$ and $\bN_G$ and recall the classical Segal machine of \cite{Shim, MMO}, adding tilde to the notation to distinguish this variant from our  variant defined using smash products.

\begin{defn}\mylabel{mach1} For an $\sF_G$-$G$-space $Y$, define $\shimtimes Y$ to be the $G$-spectrum associated to 
the $\sW_G$-$G$-space $ \widetilde{B}^{\bN_G} {Y}$, so that
\[ (\shimtimes Y)(V) = ( \widetilde{B}^{\bN_G}Y)(S^V) .\]
Similarly, define $\shimsma Y$ to be the $G$-spectrum associated to 
the $\sW_G$-$G$-space $B^{\bN_G}Y$, so that
\[ (\shimsma Y)(V) = (B^{\bN_G}Y)(S^V) .\]
\end{defn}

As was noted in \cite[Remark 3.6]{MMO}, unravelling definitions shows that $( \widetilde{B}^{\bN_G}Y)(A)$ is the geometric realization of a simplicial $G$-space with $q$-simplices given by wedges of half-smash products:
$$\bigvee_{\al_0, \dots, \al_q} A^{\al_q}\sma \big{(}\sF_G(\al_{q-1}, \al_q) \times \dots\times \sF_G(\al_0,\al_1)\times  Y(\al_0)\big{)}_+. $$
For comparison, $(B^{\bN_G}Y)(A)$ is the geometric realization of a simplicial $G$-space with $q$-simplices given by wedges of smash products:
\[\bigvee_{\al_0, \dots, \al_q} A^{\al_q}\sma \sF_G(\al_{q-1}, \al_q)\sma\dots\sma \sF_G(\al_0,\al_1)\sma Y(\al_0). \]

It was proven in \cite{Shim}, with an updated proof in \cite[Theorem 3.22]{MMO}, that $\shimtimes$ is indeed a Segal machine in the sense of \myref{Segma}.   The following results imply that $\shimsma$ gives another, equivalent, Segal machine.  

\begin{lem}\mylabel{proper} 
The $\sF_G$-$G$-spaces $\widetilde{B}(^{\times}\bF_G^{\bN_G}, ^{\times}\bF_G^{\bN_G}, Y)$ and 
$B(^{\sma}\bF_G^{\bN_G}, ^{\sma}\bF_G^{\bN_G}, Y)$ 
are proper for any $\sF_G$-$G$-space $Y$.
\end{lem}
\begin{proof}
By \cite[Proposition 9.13]{MMO}, the prolonged $\sW_G$-$G$-space given by the $( \widetilde{B}^{\bN_G}Y)(A)$ preserves Reedy cofibrancy, and the same argument works with $\times$ replaced by $\sma$. The conclusion follows from \myref{Reedy}.
\end{proof}
 
\begin{thm}\mylabel{smatimes1}
For $\sF_G$-$G$-spaces $Y$, there is a natural level $G$-equivalence of $G$-spectra $$\shimtimes Y\rtarr \shimsma Y. $$
 \end{thm}
 \begin{proof}  We have the following commutative triangle, where $p$ is the evident quotient map and the middle arrow $\epz$ is induced by passage to quotients from the left arrow $\epz$.  Since it is not reduced, the left corner is not an $\sF_G$-$G$-space, 
 but the map 
 $p$ is a map of $\sF_G$-$G$-spaces and it induces a map of prolongations to $\sW_G$-$G$-spaces, which in turn induces a map  
 $p\colon \shimtimes Y \rtarr \shimsma Y$ of $G$-spectra.
 \begin{equation}\label{triangle2}
\xymatrix{
B(^{\times}\bF_G^{\bN_G}, ^{\times}\bF^{\bN_G}_G,Y) \ar[r]^-{\htp} \ar[dr]_{\epz} 
&  \widetilde{B}(^{\times}\bF_G^{\bN_G}, ^{\times}\bF^{\bN_G}_G,Y)  \ar[r]^-{p}  \ar[d]_{\epz} 
& B(^{\sma}\bF_G^{\bN_G}, ^{\sma}\bF^{\bN_G}_G,Y) \ar[ld]^{\epz} \\
& Y & \\} 
\end{equation}
The left and right maps $\epz$ are level $G$-equivalences. As explained in \cite[\S3.1]{MMO}, $B(\ast, \bF_G^{\times,N},Y)$ is level $G$-contractible and the inclusion of it into $B(^{\times}\bF_G^{\bN_G}, ^{\times}\bF_G^{\bN_G},Y)$ is a level $G$-cofibration, hence the arrow labeled $\htp$ is a level $G$-equivalence.   Therefore the middle map $\epz$ and the quotient map 
$p$ are level $G$-equivalences.
Remembering (\ref{BAotimes}), \myref{proper} allows us to apply the invariance theorem of \myref{inv2} to obtain the conclusion.
\end{proof}

\begin{rem}\mylabel{why-not-N}
Consider $\sF$ and $\bN$.  We recall briefly from \cite[Remark 3.10]{MMO} why the obvious generalization of the nonequivariant machine to $\sF$-$G$-spaces $X$ does not work equivariantly.  Just as in the previous proof, we have the following diagram.
 \begin{equation}\label{triangle3}
\xymatrix{
B(^{\times}\bF^{\bN}, ^{\times}\bF^{\bN},Y) \ar[r]^-{\htp} \ar[dr]_{\epz} 
&  \widetilde{B}(^{\times}\bF^{\bN}, ^{\times}\bF^{\bN},Y)  \ar[r]^-{p}  \ar[d]_{\epz} 
& B(^{\sma}\bF^{\bN}, ^{\sma}\bF^{\bN},Y) \ar[ld]^{\epz} \\
& Y & \\} 
\end{equation}
The left and right maps $\epz$ are again level $G$-equivalences, as is the arrow labeled $\htp$, hence so are the middle arrow $\epz$ and the arrow $p$.  However, because the homotopy inverse $\et$ of $\epz$ on the left and right is not level $\SI$-equivariant, these level $G$-equivalences are not \gen-level equivalences.  Therefore, prolonging the diagram to $\sF_G$-$G$-spaces does \emph{not} result in level $G$-equivalences and the invariance theorem does \emph{not} apply, even though the bar constructions here are again proper. 
\end{rem}

\section{The symmetric Segal machines}\label{symma}

The previous section laid the groundwork for our symmetric Segal machines defined using $\SI$ and $\SI_G$. We here define these machines and compare them to each other and to the machines defined using $\bN_G$.  We emphasize that the comparison with the machine defined using $\bN_G$  gives the only proof we know that the symmetric machines are in fact Segal machines.

\subsection{The Segal machines defined using $\SI$ and $\SI_G$}\label{SISIG}  

Anticipating the proofs to follow, we offer the following new definitions of Segal machines.   Starting from $\sF$-$G$-spaces, the first of our new machines will solve the problem with $\sF$ and $\bN$ discussed in \myref{why-not-N}. Recall \myref{simp}.

\begin{defn}\mylabel{mach2} For an $\sF$-$G$-space $X$, define $\symfsma X$ to be the $G$-spectrum associated to 
the $\sW_G$-$G$-space $B^{\SI}X$, so that
\[ (\symfsma X)(V) = (B^{\SI}X)(S^V) .\]
\end{defn}

More explicitly, $(B^{\SI}X)(A)$ is the geometric realization of a simplicial $G$-space with $q$-simplices given by passage to orbits over symmetric group actions from  the $q$-simplices of $(B^{\bN}X)(A)$: the $q$-simplices are
\[\bigvee_{n_q, \dots, n_0} A^{n_q}\sma_{\SI_{n_q}} \sF(\mb{n_{q-1}}, \mb{n_q})\sma_{\SI_{n_{q-1}}}\dots
\sma_{\SI_{n_1}} \sF(\mb{n_0},\mb{n_1})\sma_{\SI_{n_0}} X(\mb{n_0}). \]
The following analogue provides an intermediary that will allow us to compare the Segal machines of
Definitions \ref{mach1} and \ref{mach2} but, surprisingly, it turns out to give a machine that is actually isomorphic to that of \myref{mach2}.

\begin{defn}\mylabel{mach3} For an $\sF_G$-$G$-space $Y$, define $\symsma Y$ to be the $G$-spectrum associated to 
the $\sW_G$-$G$-space $B^{\SI_G}Y$, so that
\[ (\symsma Y)(V) = (B^{\SI_G}Y)(S^V) .\]
\end{defn}

We summarize what we shall prove about these machines.
For any $\sF_G$-$G$-space $Y$, we have maps of $G$-spectra
$$\xymatrix@1{
\shimtimes Y  \ar[r]^-{p} &
\shimsma Y \ar[r]^-{q}   & \bS_G^{\SI_G}Y   \ar[r]^-{\iso} & \bS_G^{\SI} \bU^{\sF_G}_{\sF}Y.  \\} $$
For any $\sF$-$G$-space $X$, we have maps of $G$-spectra
$$ \xymatrix@1{
\shimtimes \bP_\sF^{\sF_G}X \ar[r]^-{p} &
\shimsma \bP_\sF^{\sF_G}X \ar[r]^-{q}   & \bS_G^{\SI_G}\bP_\sF^{\sF_G}X   \ar[r]^-{\iso} & \bS_G^{\SI} X.  \\} $$

\myref{MAIN1} will give that the maps labeled $\iso$ are indeed natural isomorphisms of $G$-spectra. \myref{MAIN2} will give that the 
middle quotient maps $q$ are level $G$-equivalences of $G$-spectra.  \myref{smatimes1} already gives that the left quotient maps $p$ 
are level equivalences of $G$-spectra. Thus the cited results show that the proposed Segal machines in Definitions \ref{mach2} and \ref{mach3} are level $G$-isomorphic and are level $G$-equivalent to the Segal machines of \myref{mach1}. 

\subsection{The proof that the symmetric bar constructions are proper}\label{properpf}

The comparison of machines depends on the invariance theorem, and to apply it we need to know that our symmetric bar constructions give proper $\sF$-$G$-spaces, just as the bar constructions of \S\ref{NNG} do (see \myref{proper}).    We only consider $\bF^{\SI}$ here, dealing with $\bF_G^{\SI_G}$ in \myref{MAIN1}.   Again recall \myref{simp}.

\begin{lem}\mylabel{Reedycof}
Let $A$ be a based $G$-space, and let $X$ be an $\sF$-$G$-space. Then the simplicial $G$-space 
$B_\sbt(A^\bullet, \bF^{\SI}, X)$ is Reedy cofibrant. 
\end{lem}
\begin{proof}
 By \cite[Lemma~1.9]{MMO}, it suffices to show that the degeneracy maps are based $G$-cofibrations. These are wedges of maps of the form 
 \[ A^n\sma_{\SI_n} \sF(\mathbf{k},\mathbf{n}) \sma_{\SI_k} X(\mathbf{k}) \rtarr
 A^n\sma_{\SI_n} \sF(\mathbf{k},\mathbf{n}) \sma_{\SI_k} \sF(\mathbf{k},\mathbf{k})\sma_{\SI_k} X(\mathbf{k}) 
 \]
that are obtained by inserting $\id_{\bf k}$ into one factor.  \cite[Appendix, Proposition 2.6]{BVbook} implies that $A^n$ is nondegenerately based as a $(G\times \SI_n)$-space.
We assume or arrange by whiskering that $X(\mathbf{k})$ is nondegenerately based as a $(G\times \SI_k)$-space. Since 
the mapping spaces of $\sF$ are all discrete, we deduce that the degeneracy maps are $G$-cofibrations from the cofibration property we see before passage
to orbits.  
\end{proof}

 \begin{prop}\mylabel{FGspacesproper}  Let $X$ be an $\sF$-$G$-space. Then
the $\sF$-$G$-space $B(\bF^\SI, \bF^\SI, X)$ is proper and thus, equivalently, the $\sF_G$-$G$-space 
$\bP_\sF^{\sF_G} B(\bF^\SI, \bF^\SI, X)$ is proper.
\end{prop}
\begin{proof}
By \cite[Lemma~1.9]{MMO}, it suffices to show that for any based simplicial $G$-set $A_*$, each of the degeneracy maps 
$s_i\colon A_n\rtarr A_{n+1}$ induces a $G$-cofibration $B(A_n^\bullet,\bF^{\SI},X) \rtarr B(A_{n+1}^\bullet,\bF^{\SI},X)$. By \myref{Reedycof}, these bar constructions are realizations of Reedy cofibrant simplicial $G$-spaces, so that by \cite[Theorem 1.11]{MMO} it suffices to show that each map 
\[ A_n^\bullet(\bF^{\SI})^q X \rtarr A_{n+1}^\bullet(\bF^{\SI})^q X\]
is a $G$-cofibration. Taking $q=1$ for simplicity, this map is a wedge over $k$ and $m$ of maps
\begin{equation}\label{eq:DegenCofib}
A_n^{k} \sma_{\SI_k} \sF({\bf m, k}) \sma_{\SI_m} Y(\mathbf{m})
\xrtarr{s_i\sma \id \sma \id}
 A_{n+1}^{k} \sma_{\SI_k} \sF(\mathbf{m},\mathbf{k})\sma_{\SI_m} Y(\mathbf{m}) 
 \end{equation}
 Since $s_i\colon A_n \rtarr A_{n+1}$ is an injection of (discrete) sets, it follows that the function 
 $s_i\colon A_n^{k}\rtarr A_{n+1}^{k}$
is a $(G\times\SI_k)$-cofibration, so that the map \eqref{eq:DegenCofib} is a $G$-cofibration by \cite[Appendix, Lemma 2.3]{BVbook}.
\end{proof}
 
\subsection{The comparison of $\bS_G^{\SI}$ and $\bS_G^{\SI_G}$}\label{comparison}

The commutative diagram of inclusions of categories
\[ \xymatrix{
\SI \ar[r] \ar[d] & \SI_G \ar[d] \\
\sF  \ar[r] &  \sF_G}
\]
gives rise to a commutative diagram of forgetful functors
\[ \xymatrix{
\Fun(\SI, \sT_G) & \Fun(\SI_G,\sT_G) \ar[l]_-{\bU^{\SI_G}_\SI} \\
\Fun(\sF, \sT_G)  \ar[u]^{\bU^\sF_\SI}&  \Fun(\sF_G, \sT_G) \ar[u]_{\bU^{\sF_G}_{\SI_G}} \ar[l]^-{\bU^{\sF_G}_\sF}}
\]
The left adjoints given by categorical tensor products then make the following diagram 
commute up to natural isomorphism. 
\begin{equation}\label{square1} \xymatrix{
\Fun(\SI, \sT_G) \ar[r]^-{\bP_\SI^{\SI_G}} \ar[d]_-{\bP_\SI^\sF} & \Fun(\SI_G,\sT_G) \ar[d]^-{\bP_{\SI_G}^{\sF_G}}  \\
\Fun(\sF, \sT_G)  \ar[r]_-{\bP_\sF^{\sF_G}} &  \Fun(\sF_G, \sT_G) \ultwocell<\omit>{<0> \cong \,} }
\end{equation}

The proof of the following result is identical to that for the pair $(\bP_\sF^{\sF_G}, \bU^{\sF_G}_\sF)$ given in \cite[Theorem 2.30]{MMO}. 
While we no longer have Segal maps in the top pair of functor categories in (\ref{square1}), the respective notions of \gen-level equivalence and level
$G$-equivalence still make sense.

\begin{thm}  The adjoint pair $(\bP_\SI^{\SI_G},\bU^{\SI_G}_\SI)$ gives an equivalence of categories.  A map $f\in \Fun (\SI,\sT_G)$ 
is an \gen-level equivalence if and only if $\bP_\SI^{\SI_G}f$ is a level $G$-equivalence.
\end{thm} 

\begin{rem}\mylabel{Nfails}
The proof of \cite[Theorem 2.30]{MMO} is given for the pair of adjoint functors $(\bP_\PI^{\PI_G},\bU^{\PI_G}_\PI)$. The same proof works with $\PI$ replaced by $\sF$, $\SI$, or any subcategory of $\sF$ that contains $\SI$. The proof does not work with $\PI$ replaced by $\bN$. Since
there are no non-identity morphisms in $\bN_G$,  $\bP_\bN^{\bN_G}X(\al)$ is a point when $\al$ is a non-trivial $G$-set.
\end{rem}

Since $(\bU^{\SI_G}_\SI, \bP_\SI^{\SI_G})$ and $( \bU^{\sF_G}_\sF, \bP_\sF^{\sF_G})$ are equivalences of categories, the following square also commutes up to natural isomorphism

\begin{equation}\label{square2} \xymatrix{
\Fun(\SI, \sT_G) \ar[r]^-{\bP^{\SI_G}_\SI} & \Fun(\SI_G,\sT_G)  \\
\Fun(\sF, \sT_G)  \ar[u]^{\bU^\sF_\SI} \ar[r]_-{\bP^{\sF_G}_\sF}
  &  \Fun(\sF_G, \sT_G) \ar[u]_{\bU^{\sF_G}_{\SI_G}} \ultwocell<\omit>{<0>  \cong \,} }  \end{equation}

Combining squares (\ref{square1}) and (\ref{square2}) and remembering that our monads are composites of the form $\bU\bP$, we obtain 
a natural isomorphism 
\[\xymatrix{
\Fun(\SI, \sT_G) \ar[d]_-{\bF^{\SI}} \ar[r]^-{\bP^{\SI_G}_\SI} & \Fun(\SI_G,\sT_G)\ar[d]^-{\bF_G^{\SI_G}}  \\
\Fun(\SI, \sT_G) \ar[r]_-{\bP^{\SI_G}_\SI}   & \Fun(\SI_G,\sT_G)  \ultwocell<\omit>{<0>  \cong \,} 
}\]      

This isomorphism implies the following precise comparison of the symmetric machines defined in Definitions \ref{mach2} and \ref{mach3}.

\begin{thm}\mylabel{MAIN1}
Let $X$ be an $\sF$-$G$-space. Then there is a natural isomorphism of $\sF_G$-$G$-spaces
\begin{equation}\label{LR} 
B(\bF_G^{\SI_G}, \bF_G^{\SI_G}, \bP_\sF^{\sF_G}X)\cong \bP_\sF^{\sF_G} B(\bF^{\SI}, \bF^{\SI}, X)
\end{equation}
and $B(\bF_G^{\SI_G}, \bF_G^{\SI_G}, \bP_\sF^{\sF_G}X)$ is proper. Applying  $\bP_{\sF_G}^{\sW_G}$ 
and restricting to spheres, there is a natural isomorphism of $G$-spectra
$$ \bS_G^{\SI_G} \bP_\sF^{\sF_G}X \iso \bS_G^{\SI} X.  $$
Equivalently, for an $\sF_G$-$G$-space $Y$, there is a natural isomorphism of $G$-spectra
$$ \bS_G^{\SI_G} Y \iso \bS_G^{\SI} \bU_\sF^{\sF_G}Y.  $$
\end{thm}
\begin{proof}
To be pedantically precise, the left hand side of (\ref{LR}) is really 
 $$\bP_{\SI_G}^{\sF_G} B(\bF_G^{\SI_G}, \bF_G^{\SI_G}, \bU^{\sF_G}_{\SI_G} \bP_\sF^{\sF_G}X) \cong \bP_{\SI_G}^{\sF_G} B(\bF_G^{\SI_G}, \bF_G^{\SI_G}, \bP_\SI^{\SI_G}\bU^{\sF}_{\SI} X).$$
Since $\bF_G^{\SI_G}\bP_\SI^{\SI_G}\cong \bP_\SI^{\SI_G}\bF^\SI$,
$$B(\bF_G^{\SI_G}, \bF_G^{\SI_G}, \bP_\SI^{\SI_G}\bU^{\sF}_{\SI} X)\cong \bP_\SI^{\SI_G}B(\bF^\SI, \bF^\SI, \bU^{\sF}_{\SI} X).$$
Applying $ \bP_{\SI_G}^{\sF_G} $, we get
$$ \bP_{\SI_G}^{\sF_G} \bP_\SI^{\SI_G}B(\bF^\SI, \bF^\SI, \bU^{\sF}_{\SI} X)\cong \bP_\sF^{\sF_G} \bP_\SI^\sF B(\bF^\SI, \bF^\SI, \bU^{\sF}_{\SI} X),$$
which is exactly what the right hand side of (\ref{LR}) really means.  The statement about properness follows, since the isomorphism is obtained by applying geometric realization to an isomorphism of simplicial $G$-spaces and the right side of (\ref{LR}) is proper by \myref{FGspacesproper} . The last statements follow directly. 
\end{proof}

We find these isomorphisms quite remarkable.  If we write out the definitions explicitly and try to compare the bar constructions explicitly,
we find the task quite forbidding.  We emphasize that, as a result of \myref{Nfails}, the analogous theorem does \emph{not} hold if 
we replace $\SI$ and $\SI_G$ with $\bN$ and $\bN_G$.

\subsection{The comparison of $\bS_G^{\bN_G}$ and $\bS_G^{\SI_G}$}\label{finalcomp}
Arguing as in the previous section, the inclusion of domain categories 
$\bN_G\rtarr \SI_G$ gives rise to  a natural transformation (not an isomorphism)
\[\xymatrix{
\Fun(\bN_G, \sT_G) \ar[d]_-{\bF_G^{\bN_G}} \ar[r]^-{\bU^{\bN_G}_{\SI_G}} & \Fun(\SI_G,\sT_G)\ar[d]^-{\bF_G^{\SI_G}}  \\
\Fun(\bN_G, \sT_G) \ar[r]_-{\bU^{\bN_G}_{\SI_G}}  \urtwocell<\omit>{<0> \, q} 
 & \Fun(\SI_G,\sT_G)  }\]      
For an $\sF_G$-$G$-space $Y$, $q$ induces a map of $\sF_G$-$G$-spaces 
$$q\colon B(\bF_G^{\bN_G},\bF_G^{\bN_G}, Y) \rtarr B(\bF_G^{\SI_G}, \bF_G^{\SI_G}, Y).$$
This gives a formal construction of the natural quotient map $q$, and the following diagram of $\sF_G$-$G$-spaces commutes.
\begin{equation}\label{keydia}
\xymatrix{
B(\bF_G^{\bN_G},\bF_G^{\bN_G}, Y)\ar[dr]_-\epz \ar[rr]^-{q} & & \ar[dl]^-\epz B(\bF_G^{\SI_G}, \bF_G^{\SI_G}, Y)\\
&  Y &\\}
\end{equation}
The maps $\epz$ are level $G$-equivalences, hence so is $q$.  Since the bar constructions are proper, by \myref{proper} and \myref{MAIN1}, the invariance theorem applies to give the following conclusion. 

\begin{thm}\mylabel{MAIN2}
For any $\sF_G$-$G$-space $Y$, the map $q$  of (\ref{keydia}) is a level $G$-equivalence of $\sF_G$-$G$-spaces.
Therefore, applying  $\bP_{\sF_G}^{\sW_G}$ and restricting to spheres, $q$ induces a level $G$-equivalence of $G$-spectra
$$ \shimsma Y\rtarr \bS_G^{\SI_G}Y.$$
\end{thm}

Taking $Y= \bP_{\sF}^{\sF_G}X$ for an $\sF$-$G$-space $X$ and combining Theorems \ref{smatimes1}, \ref{MAIN1} and \ref{MAIN2}, we see that the Segal machine $ \shimtimes \bP_\sF^{\sF_G}X $ of \cite{Shim, MMO} is equivalent to our new symmetric Segal machine  
$\bS_G^{\SI}X$.

\begin{rem}\mylabel{PIPIG}  The reader of \cite{MMO} may notice that when dealing with the operadic infinite loop space machine there we constructed appropriate monads using $\PI$ and $\PI_G$ 
rather than $\SI$ and $\SI_G$.   We could have used $\PI$ and $\PI_G$ here too.  However, it would be finicky to adapt the
proof of \myref{FGspacesproper} to show that the bar constructions that are relevant here are proper and thus to justify application of the invariance theorem using the resulting monads.  The variant machines 
that would be obtained offer no apparent advantages over those already constructed.
\end{rem}

\section{Symmetric monoidal properties of the symmetric Segal machine}\label{SymMon}

We recall the monoidal structures on the categories of $\sF$-$G$-spaces and of orthogonal $G$-spectra and 
state the theorem about the symmetric monoidal Segal machine in \S\ref{state}.  For definiteness, we 
take $\bS_G$ to be $\symfsma$ in this section and remove the superscripts from the notation.  We prove that 
$\bS_G$ is lax monoidal  in  \S\ref{monoidal}.  This proof can be adapted for any of the other variants. We prove that 
$\shim$ is lax symmetric monoidal in \S\ref{symmon}. There it is essential to use the symmetric machine.

\subsection{Recollections about product structures}\label{state} 
The category $\sF$ is permutative with respect to the smash product $\sma$. Thus, we can define a symmetric 
monoidal product on $\sF$-$G$-spaces by using Day convolution; see e.g. \cite[\S21]{MMSS}. For $\sF$-$G$-spaces $X$ and $Y$, we have the evident levelwise external smash product 
$X \barwedge Y$, defined as the composite\footnote{$\sF\times \sF$ and $G\sT \times G\sT$ are understood in the enriched sense, so that hom objects are smash products of the hom objects of the two variables.} 
\[
\sF \times \sF \xrightarrow{X \times Y} G\sT \times G\sT \xrightarrow{\sma} G\sT.
\]
We define $X\sma Y$ as the left Kan extension indicated in the diagram
\[ \xymatrix{
\sF\times \sF \ar[d]_{\sma} \ar[r]^-{X\barwedge Y} & G\sT. \\
\sF \ar[ur]_{X\sma Y} } 
\]
It is characterized by the adjunction 
\begin{equation}\label{Day}
{\Fun}(\sF,G\sT)(X\sma Y, Z) \iso {\Fun}(\sF\times \sF,G\sT)(X\barwedge Y,Z\circ \sma),
\end{equation}
where $Z\in {\Fun}( \sF,G\sT)$.
Its unit is the $\sF$-$G$-space $\mathbf{I}$ given by the inclusion 
$\sF\subset \sT\subset G\sT$.
As motivation, pairings $X\barwedge Y\rtarr Z\circ \sma$ of $\sF$-$G$-spaces appear ubiquitously
in nature, as was already noted nonequivariantly in \cite{Maypair}.\footnote{Regrettably, its author did not then know about Day convolution,
which converts such pairings to maps of $\sF$-$G$-spaces.}

The category of orthogonal $G$-spectra is also symmetric monoidal \cite[II.3.1]{MM}, and its smash product
is defined similarly.  Recall the definition of orthogonal $G$-spectra and of the external smash
product $\barwedge$ from \cite[Definition 1.18]{MMO}. Recall too that the structure maps of $X$ give a map 
of $\sI_G$-$G$-spaces $\si\colon X\barwedge S\rtarr X\com \oplus$.  For $G$-spectra $X$ and $Y$, 
$X\sma Y$ starts from the external smash product $X\barwedge Y$ and its left Kan extension indicated in the diagram
\[ \xymatrix{
\sI_G\times \sI_G \ar[d]_{\oplus} \ar[r]^-{X\barwedge Y} & \sT_G. \\
\sI_G \ar[ur]_{X\sma_{\sI_G} Y} \\} 
\]
This left Kan extension is characterized by the adjunction 
\begin{equation}\label{Day2}
{\Fun}(\sI_G,\sT_G)(X\sma_{\sI_G} Y, Z) \iso {\Fun}(\sI_G\times \sI_G,\sT_G)(X\barwedge Y,Z \com \oplus),
\end{equation}
where $Z\in {\Fun}( \sI_G,\sT_G)$.
In particular, the structure map $\si$ induces a map 
\[X\sma_{\sI_G} S \rtarr X.\]
 The smash product $X\sma Y$ 
is the coequalizer displayed in the diagram
\begin{equation}\label{coeq}
\xymatrix@1{  X\sma_{\sI_G} S\sma _{\sI_G} Y \ar@<.5ex>[r] \ar@<-.5ex>[r] & X\sma_{\sI_G} Y \ar[r] & X\sma Y,}\\
\end{equation}
where the two arrows are induced by the action of $S$ on $X$ and $Y$ (using $S\sma_{\sI_G} Y\iso Y\sma_{\sI_G} S)$.
The action of $S$ on $X$ or $Y$ induces the required action of $S$ on $X\sma Y$.
For the unit property, $X\sma_{\sI_G} S$ becomes isomorphic to $X$ on passage to coequalizers.

\begin{defn} For $G$-spectra  $X$, $Y$, and $Z$, a pairing $f\colon X\sma Y\rtarr Z$ is a 
map of $G$-spectra.
\end{defn}  

\begin{thm}\mylabel{SymMonThm}
The functor $\shim$ from $\sF$-$G$-spaces to orthogonal $G$-spectra is lax symmetric monoidal.
Therefore a pairing $f\colon X\sma Y \rtarr Z$ of $\sF$-$G$-spaces functorially determines a
pairing $\shim f\colon \shim X\sma \shim Y\rtarr \shim Z$ of $G$-spectra.
\end{thm}

\begin{rem}\mylabel{conceptuallax} The equivariant version of \cite[Proposition 3.3]{MMSS} shows that the conceptual variant of the Segal
machine, namely the functor $\bU_{G\sS}\bP_{\sF}^{\sW_G}$ of \cite[Definition 2.20]{MMO}, is also a lax symmetric monoidal 
functor from $\sF$-$G$-spaces to orthogonal $G$-spectra.  However, we are more interested in the homotopically better behaved functor $\shim$.
\end{rem}

We prove \myref{SymMonThm} in two steps.  We first prove that $\shim$ is lax monoidal
and then prove that it is symmetric.  We repeat that the symmetric variant of the bar construction used in
our construction of $\shim$ is vital for the second part.

\subsection{The proof that $\shim$ is lax monoidal}\label{monoidal}

Let $X$ and $Y$ be $\sF$-$G$-spaces.
We need maps of $G$-spectra $\epz \colon S_G \rtarr \shim(\mathbf{I})$ and
\[ \varphi \colon \shim(X)\sma \shim(Y) \rtarr \shim(X\sma Y), \] 
the latter natural in $X$ and $Y$, such that the following three
diagrams commute:
\[
\xymatrix{
S_G \sma \shim(X)\ar[r]^-{\epz \sma \id} \ar[d]_{\cong} & \shim(\mathbf{I}) \sma \shim(X)\ar[d]^{\varphi}\\
\shim(X)  &  \shim(\mathbf{I}\sma X),\ar[l]^{\cong}
}
\quad
\xymatrix{
\shim(X) \sma S_G \ar[d]_{\cong} \ar[r]^-{\id \sma \epz} & \shim(X) \sma \shim(\mathbf{I})\ar[d]^{\varphi}\\
\shim(X)  &  \shim(X\sma \mathbf{I}),\ar[l]^{\cong}
}
\]

\[
\xymatrix{
\shim(X)\sma \shim(Y) \sma \shim(Z) \ar[r]^-{\varphi \sma \id} \ar[d]_{\id \sma \varphi} &\shim(X\sma Y) \sma \shim(Z) \ar[d]^{\varphi}\\
\shim(X)\sma \shim(Y\sma Z) \ar[r]_{\varphi} &  \shim(X\sma Y \sma Z).
}
\]
In the last diagram we have omitted the associativity isomorphisms for both $\sF$-$G$-spaces and orthogonal $G$-spectra. 

Note that $S^V$ includes into the 0-simplices of the bar construction $\barc{(S^V)^\bullet}{\bF^{\SI}}{\mb{I}}$ by mapping into the component $S^V \cong S^V \sma \mb{1}$. It is straightforward to check that this map gives a map of spectra $\epz \colon S_G \rtarr \shim(\mb{I})$.

To construct $\varphi$, we define a map
\begin{equation}\label{varphi} B\big{(} (S^V)^{\bullet}, \bF^{\SI}, X \big{)} \sma B\big((S^W)^{\bullet},\bF^{\SI},Y\big)
\rtarr \big((S^{V\oplus W})^{\bullet}, \bF^{\SI}, X\sma Y\big)\end{equation}
for $X$, $Y$ $\sF$-$G$-spaces, and $V$, $W$ in $\sI_G$. We will define the map on wedge summands at the simplicial level, before modding out by the symmetric groups. As usual, we use lexicographic ordering to identify the smash product $\mb m\sma \mb n$ of finite based sets with 
$\bf{mn}$.  This is where asymmetry enters.  To simplify notation, for a tuple $(n_0,\dots,n_q)$, write 
\begin{equation}\label{Fmnotn}
\sF(\ul{\mb{n}}) = \sF(\mb{n_{q-1}},\mb{n_q})\sma \cdots \sma \sF(\mb{n_0},\mb{n_1}).
\end{equation}

Given a second tuple $(m_0,\dots,m_q)$, write 
\begin{equation}\label{Fmnotn2}
\sF(\ul{\mb{m}}) \sma \sF(\ul{\mb{n}}) = \sF(\mb{m_{q-1}},\mb{m_q})\sma \sF(\mb{n_{q-1}},\mb{n_q}) \sma \cdots \sma \sF(\mb{m_0},\mb{m_1})\sma \sF(\mb{n_0},\mb{n_1})
\end{equation}
and
\begin{equation}\label{Fmnotn3}
\sF( \ul{\mb{m}}\sma \ul{\mb{n}}) =  \sF(\mb{m_{q-1}n_{q-1}}, \mb{m_{q}n_{q}}) \sma \cdots \sma
 \sF(\mb{m_{0}n_{0}}, \mb{m_{1}n_{1}}).
\end{equation}

On $q$-simplices, the map in (\ref{varphi}) is defined, after passing to orbits, as the composite
\begin{equation}\label{phi_simplices}
\xymatrix{
(S^V)^{m_q} \sma \sF(\ul{\mb{m}}) \sma X(\mb{m_0}) \sma 
(S^W)^{n_q} \sma \sF(\ul{\mb{n}}) \sma Y(\mb{n_0}) \ar[d]\\
(S^V)^{m_q} \sma (S^W)^{n_q} \sma \sF(\ul{\mb{m}}) \sma \sF(\ul{\mb{n}}) \sma X(\mb{m_0}) \sma Y(\mb{n_0})\ar[d]\\
(S^{V\oplus W})^{m_q n_q} \sma \sF( \ul{\mb{m}}\sma \ul{\mb{n}}) \sma (X\sma Y)(\mb{m_0} \sma \mb{n_0}).
}
\end{equation}
The first map is a shuffle map. The second map is the smash product of the map 
\[(S^V)^{m_q}\sma (S^W)^{n_q} \rtarr (S^{V\oplus W})^{m_q n_q}\]
that sends $(v_1,\dots,v_m)\sma (w_1,\dots, w_n)$ to $(v_i\sma w_j)$ in the $(i,j)$th-coordinate, 
the map induced by $\sma$ on $\sF$, and the map
\[X(\mb{m_0})\sma Y(\mb{n_0}) \rtarr (X\sma Y)(\mb{m_0} \sma \mb{n_0})\]
induced by the universal property of Day convolution. 

One can easily check that these maps commute with the simplicial maps and are compatible with the symmetric group actions that are used in the passage to orbits in the definition of $B^\SI$, thus giving a pairing
of $G$-functors $\sI_G\times \sI_G\rtarr \sT_G$.  By the universal property of left Kan extension, there results a map
\[\shim(X)\sma_{\sI_G} \shim(Y) \rtarr \shim (X\sma Y). \]
of $\sI_G$-$G$-spaces.  It is routine to check that this map factors through the coequalizer in (\ref{coeq}), and this gives the required map
\[ \varphi\colon \shim(X)\sma \shim(Y) \rtarr \shim (X\sma Y) \] 
of $G$-spectra.
From here, it is straightforward to check that the three diagrams do in fact commute.  For the unit diagrams, the 
essential point is just that $S^0$ is the unit for $\sma$ in $\sT_G$. For the associativity diagram, the essential
point is just that $\sma$ on (nondegenerately based compactly generated) $G$-spaces is associative (up to canonical isomorphism).

\subsection{The proof that $\shim$ is lax symmetric monoidal}\label{symmon}

In the previous section we could have worked just as well with the machine defined in \cite{MMO} in terms of
the ordinary bar construction, but the resulting lax monoidal machine would not be lax symmetric monoidal.
In fact, the dichotomy is already present nonequivariantly, where Woolfson's variant \cite{Woolf} of the Segal 
machine, which is implicitly defined using the ordinary bar construction, is lax monoidal but not lax symmetric monoidal.

We must prove that the following diagram commutes.
\[  \xymatrix{
\bS_G(X)\sma \bS_G(Y) \ar[r]^-{\varphi} \ar[d]_{\ta} &  \bS_G(X\sma Y) \ar[d]^{\bS_G \ta} \\
\bS_G(Y)\sma \bS_G(X) \ar[r]_-{\varphi}  &  \bS_G(Y\sma X).\\} \]

In general, for orthogonal $G$-spectra $X$, $Y$, and $Z$ with maps $f\colon X\sma Y \rtarr Z$ and $e\colon Y\sma X \rtarr Z$, the 
interpretation in terms of the external smash product for the diagram
\[
\xymatrix{
X\sma Y \ar[dr]^f\ar[d]_{\tau}\\
Y\sma X \ar[r]_e & Z
}
\]
to commute is that the diagram
\[
\xymatrix{
X(V)\sma Y(W) \ar[r]^f  \ar[d]_{\tau} & Z(V\oplus W) \ar[d]^{Z(\tau)}\\
Y(W)\sma X(V) \ar[r]_e & Z(W\oplus V) 
}
\] 
commutes for all $V$ and $W$. Therefore, with the horizontal arrows defined as in the previous section, the diagram that we are 
trying to show commutes translates to the commutativity of the diagram
\[ \xymatrix{
B((S^V)^{\bullet}, \bF^{\SI}, X)\sma B((S^W)^{\bullet}, \bF^{\SI}, Y)  \ar[dd]_{\ta} \ar[r] & B((S^{V\oplus W})^{\bullet}, \bF^{\SI}, X\sma Y)
\ar[d]^{B(\id,\id,\ta)} \\
& B((S^{V\oplus W})^{\bullet}, \bF^{\SI}, Y\sma X) \ar[d]^{B(\tau^{\bullet}, \id, \id)}\\
B((S^W)^{\bullet}, \bF^{\SI}, Y)\sma B((S^V)^{\bullet}, \bF^{\SI}, X) \ar[r] & B((S^{W\oplus V})^{\bullet}, \bF^{\SI}, Y\sma X), \\} \]
which can be rewritten as
\[ \xymatrix{
B((S^V)^{\bullet}, \bF^{\SI}, X)\sma B((S^W)^{\bullet}, \bF^{\SI}, Y)  \ar[d]_{\ta} \ar[r] & B((S^{V\oplus W})^{\bullet}, \bF^{\SI}, X\sma Y)
\ar[d]^{B(\tau^\bullet,\id,\ta)} \\
B((S^W)^{\bullet}, \bF^{\SI}, Y)\sma B((S^V)^{\bullet}, \bF^{\SI}, X) \ar[r] & B((S^{W\oplus V})^{\bullet}, \bF^{\SI}, Y\sma X). \\} \]

That diagram does not commute with $\bF^{\SI}$ replaced by $\bF^{\bN}$ , but it does commute as written.
The quotienting by permutations that appears in the construction of $\bF^{\SI}$ rectifies the noncommutativity that was introduced by the choice of lexicographic ordering on smash products.
To see the intuition, observe that the diagram for the actual tensor 
product of functors over $\sF$ rather than the bar construction clearly commutes (as is needed to verify that the conceptual Segal machine is lax symmetric monoidal, 
as claimed in the introduction).  That diagram takes the form
\[ \xymatrix{
((S^V)^{\bullet}\otimes_{\sF}X)\sma ((S^W)^{\bullet}\otimes_{\sF}Y) \ar[r] \ar[d]_{\ta} 
& (S^{V\oplus W})^{\bullet}\otimes_{\sF}(X\sma Y) \ar[d]^{\ta\otimes\ta}\\
((S^W)^{\bullet}\otimes_{\sF}Y)\sma ((S^V)^{\bullet}\otimes_{\sF}X) \ar[r]
& (S^{W\oplus V})^{\bullet}\otimes_{\sF}(Y\sma X).\\}
\]

Expanding out the definitions of the maps $\varphi$ on $q$-simplices given in (\ref{phi_simplices}), we see that going clockwise we get
\[
\xymatrix{
(S^V)^{m_q} \sma \sF(\ul{\mb{m}}) \sma X(\mb{m_0}) \sma (S^W)^{n_q} \sma \sF(\ul{\mb{n}}) \sma Y(\mb{n_0})\ar[d]\\
(S^{W\oplus V})^{m_q n_q} \sma \sF(\ul{\mb{m}}\sma \ul{\mb{n}}) \sma (Y\sma X)(\mb{m_0\sma n_0}),}\]
while going counterclockwise we get 
\[
\xymatrix{
(S^V)^{m_q} \sma \sF(\ul{\mb{m}}) \sma X(\mb{m_0})) \sma (S^W)^{n_q} \sma \sF(\ul{\mb{n}}) \sma Y(\mb{n_0})\ar[d]\\
(S^{W\oplus V})^{n_q  m_q} \sma \sF(\ul{\mb{n}}\sma \ul{\mb{m}}) \sma (Y\sma X)(\mb{n_0\sma m_0}).} \]

When we pass to orbits over the groups $\SI_{m_i n_i}$ and consider the evident permutations $\tau_i \colon m_i \sma n_i \rtarr n_i \sma m_i$ in the three variables of our two-sided bar constructions, we see that the diagram commutes after passage to quotienting by the symmetric groups. To see what is happening on the first two variables, 
just observe that the following general diagram commutes for based $G$-spaces $A,B,C,D$
\[
\xymatrix{
\sT_G(A,B) \sma \sT_G(C,D) \ar[r]^-{\ta} \ar[d]_{\sma} &\sT_G(C,D) \sma \sT_G(A,B) \ar[d]^{\sma}\\
\sT_G(A\sma C, B\sma D) \ar[r]_{\sT_G(\ta,\ta)} & \sT_G(C\sma A, D\sma B).
}
\]
To see what is happening on the third variable, we note that the symmetry isomorphism $\tau \colon X\sma Y \to Y \sma X$ of the Day convolution on $\sF$-$G$-spaces, is defined as the morphism corresponding via (\ref{Day}) to the composite
\[X(\mb{m})\sma Y(\mb{n}) \xrightarrow{\tau} Y(\mb{n})\sma X(\mb{m}) \rightarrow (Y\sma X)(\mb{n}\sma\mb{m})\xrightarrow{(Y\sma X)(\tau)} (Y\sma X)(\mb{m}\sma \mb{n}),\] where the unlabeled arrow corresponds to the identity of $Y\sma X$ under (\ref{Day}). Thus the diagram
\[
\xymatrixcolsep{2cm}\xymatrix{
X(\mb{m})\sma Y(\mb{n}) \ar[rr]^-{\ta}  \ar[d] & & Y(\mb{n}) \sma X(\mb{m}) \ar[d] \\
(X\sma Y)(\mb{m \sma n}) \ar[r]_{\ta} & (Y \sma X)(\mb{m\sma n}) \ar[r]_-{(Y\sma X)(\ta)} & (Y\sma X)(\mb{n\sma m})} \]
commutes.

\section{Examples}\label{homotopical}
We give two elementary examples where the symmetric monoidal property of the Segal machine appears naturally. 
We first give a symmetric monoidal functor from based $G$-spaces to $\sF$-$G$-spaces that gives a multiplicative version
of the Barratt-Priddy-Quillen theorem on application of the Segal machine. We then give a lax symmetric monoidal functor from 
abelian $G$-groups to $\sF$-$G$-spaces that gives rise to genuine ring, module, and algebra Eilenberg-Mac\,Lane $G$-spectra 
on application of the machine.  A brief final section shows that the Segal machine preserves homotopies.

\subsection{Suspension $G$-spectra}\label{suspension}

Any equivariant infinite loop space machine should encode a version of the Barratt-Priddy-Quillen theorem
expressing suspension $G$-spectra as outputs of input to the machine.  As explained in \cite{GM3},
this is true of the operadic machine, where free $E_{\infty}$ $G$-spaces give rise to suspension
$G$-spectra.  In this section we explain the very different way that suspension $G$-spectra appear
in the Segal machine.   We give a monoidal equivalence between the suspension $G$-spectrum functor 
and the composite of the Segal machine with an elementary functor from $G$-spaces to $\sF$-$G$-spaces.  
For definiteness, we again take $\bS_G$ to mean $\bS_G^{\SI}$ in this section.

Recall that $\SI^{\infty}_G$ is the functor from based $G$-spaces $G\sT$  to orthogonal $G$-spectra $G\sS$
that sends $X$ to $\{\SI^V X\}$.  It is left adjoint to the $0$th $G$-space functor $(-)_0$, given by evaluation at $S^0$.  The following result is well-known, but 
is not well documented in the literature.

\begin{lem}  There is a natural isomorphism of $G$-spectra
\begin{equation}\label{finick1}
 \SI^{\infty}_G X \sma \SI^{\infty}_G Y \rtarr \SI^{\infty}_G (X\sma Y),
 \end{equation}
and the functor $\SI^{\infty}_G$ from based $G$-spaces to $G$-spectra is symmetric monoidal.
 \end{lem}
\begin{proof}[Sketch] For $G$-spaces $X$ and $Y$, the obvious isomorphisms 
\[ \SI^V X \sma \SI^W Y \iso \SI^{V\oplus W}(X\sma Y) \]
specify a natural isomorphism
\[ \SI^{\infty}_G X\barwedge  \SI^{\infty}_G Y \iso  \SI^{\infty}_G (X\sma Y) \com \oplus \]
of $G$-functors $\sI_G\times \sI_G \rtarr \sT_G$. By the universal property of left Kan extension, there results a natural map 
of $G$-functors $\sI_G\rtarr \sT_G$
\begin{equation}\label{finick2}
\SI^{\infty}_G X \sma_{\sI_G}  \SI^{\infty}_G Y \rtarr \SI^{\infty}_G (X\sma Y). 
\end{equation}
The smash product $\SI^{\infty}_G X \sma  \SI^{\infty}_G Y$ is obtained by coequalizing the actions
of the sphere $G$-spectrum on $\SI^{\infty}_G X$ and $\SI^{\infty}_G Y$, and the map (\ref{finick2})
factors through the coequalizer to give the map (\ref{finick1}).  By a check of definitions, the $0$th 
$G$-space of $\SI^{\infty}_G X \sma  \SI^{\infty}_G Y$ is homeomorphic to $X\sma Y$, 
and the adjoint of this homeomorphism is a natural map
\begin{equation}\label{finick3}
 \SI^{\infty}_G (X\sma Y) \rtarr \SI^{\infty}_G X\sma \SI^{\infty}_G Y.
\end{equation}
The maps (\ref{finick1}) and (\ref{finick3}) are inverse isomorphisms of $G$-spectra.
Alternatively, one can check directly that $\SI^\infty_G (X\sma Y)$ satisfies the universal property of the coequalizer
that defines $ \SI^{\infty}_G X\sma \SI^{\infty}_G Y$.  The last statement is clear from the construction of the 
isomorphism.
\end{proof}

Slightly generalizing a notation used before, write $Y^{\bullet}$ for the contravariant functor  $\sF\rtarr G\sT$ given by the powers 
$Y^n =\sT_G({\bf n},Y)$ of any based $G$-space $Y$.  Analogously, 
write $^{\bullet}\! X$ for the covariant functor $\sF\rtarr G\sT$ given by the $n$-fold wedges $^{n}\! X = {\bf{n}}\sma X$  
of a based $G$-space $X$.  Since  $^n\!S^0 = \mathbf{n}\iso \sF(\mathbf{1},\mathbf{n})$, $^\bullet\!S^0 \iso \sF(\mathbf{1},-)$ 
is the unit $\sF$-$G$-space previously denoted by $\bf I$.  Note that $^\bullet\! X$ is very far from being 
\gen-special, or even naively special, and bears no obvious relationship to free $E_{\infty}$ $G$-spaces.  We shall prove that application 
of the functor $\bS_G$ to these $\sF$-$G$-spaces gives a functor from $G$-spaces to $G$-spectra that is 
monoidally equivalent to  $\SI^{\infty}_G$. 

\begin{rem} With the notations of \cite[Definition 1.3]{MMSS}, we have
\[ ^{n}\! X  = \mb{n}\sma X = \sF(\mathbf 1,\mathbf{n})\sma X = (F_1X)(\mathbf n), \]
and the functor $F_1 =  ^\bullet\! \! (-)$ from $G$-spaces to $\sF$-$G$-spaces is left adjoint to the
functor $\mathbf{Ev}_1$ specified by evaluation at $\mathbf 1$.  
\end{rem}

\begin{lem} The functor $F_1 =^{\bullet}\!\!(-)$ is strong symmetric monoidal.
\end{lem}
\begin{proof} This is implied by \cite[Lemma~1.8]{MMSS}, which specializes to give the required natural isomorphism
$F_1 X \sma F_1 Y \iso F_1(X\sma Y)$.
\end{proof}

\begin{thm}\mylabel{spacecomp} For based $G$-spaces $X$, there is a natural weak equivalence
$$\mu\colon \SI^{\infty}_G X \rtarr \bS_G(^{\bullet}\! X)$$ 
of $G$-spectra, and $\mu$ is a monoidal natural transformation.
\end{thm}

To prove the theorem conceptually, we put it in a more general context, starting with
the following two observations.

\begin{lem}\mylabel{coequal}   The $G$-space $Y^{\bullet}\otimes_{\sF} {^{\bullet}\!X}$ is 
$G$-homeomorphic to $Y\sma X$.  
\end{lem} 
\begin{proof}  Recall that the tensor product of functors is given by the evident coequalizer and is a quotient of the wedge over $n$ of the $G$-spaces
$$Y^n \sma {^n\! }X\iso \bigvee_{1\leq j \leq n} (Y^n\sma X).$$  Using the inclusions $\io_j\colon \mathbf 1 \rtarr \mathbf n$, 
$1\leq j\leq n$, and noting that $\io_j^*\colon Y^n \rtarr Y$ is projection on the $j$th coordinate while
${\io_j}_*\colon X\rtarr ^{n\!}X$ is inclusion of the $j$th wedge summand, we see that each wedge summand is 
identified with $Y\sma X$ on passage to the coequalizer.
\end{proof}

\begin{lem}\mylabel{catstan}
Let $\sD$ and $\sE$ be monoidal categories and let $\bL\colon \sD\rightleftarrows\sE\colon \bR$ be a monoidal adjunction, 
meaning that $\bL$ is strong monoidal and (consequently) $\bR$ is lax monoidal. Suppose that $\bH\colon \sD \rtarr \sE$ is a 
lax monoidal functor. Then a transformation
\[ \mu: \bL \Longrightarrow \bH\] 
is monoidal if and only if the adjoint transformation
\[ \hat\mu\colon  \mathrm{Id} \Longrightarrow \bR\bH\]
is monoidal.
\end{lem}

\begin{thm}\mylabel{coequalconj}  The $G$-space $B(Y^{\bullet},\bF^{\SI},^\bullet\!X)$ is naturally $G$-equivalent to $Y\sma X$.
\end{thm}
\begin{proof}
The inclusion of $Y\sma X$ in the $G$-space of $0$-simplices of the simplicial bar construction
induces a $G$-map $\et\colon Y\sma X \rtarr B(Y^{\bullet},\bF^{\SI},^{\bullet\!}X)$.  We have a 
$G$-map $ \zeta \colon B(Y^{\bullet},\bF^{\SI},^{\bullet\!}X)\rtarr Y\otimes_{\sF} X \iso Y\sma X$
obtained as usual by composition and evaluation maps.  Clearly $ \zeta \com \et = \id$.  Noting
that 
\[ B\big(Y^{\bullet},\bF^{\SI},^{\bullet\!}X\big) = B\big(Y^{\bullet},\bF^{\SI},\sF(\mathbf 1,\bullet)\sma X\big), \]
we see that the identity map $\bf 1 \rtarr \bf 1$ gives rise to an extra degeneracy operator, and
then a standard argument gives a homotopy $\et\com  \zeta \htp \id$, as we used in \S\ref{variants}.
\end{proof}

\begin{proof}[Proof of \myref{spacecomp}] 
The map $\et$ of the previous proof specializes to give 
$$\et\colon X\rtarr B((S^0)^{\bullet},\bF^{\SI},^{\bullet\!}X) = \bS_G(^{\bullet}\! X)_0.$$
We define $\mu$ to be its adjoint under the adjunction between $\SI_G^\infty$ and the $0$th space functor. Theorem \ref{coequalconj} implies that $\mu$ is a levelwise
$G$-homotopy equivalence.   It remains to prove that $\mu$ is monoidal. 

Applying \myref{catstan} with $\bL=\SI^{\infty}_G$, $\bR=(-)_0$, and $\bH=\bS_G(^{\bullet}(-))$, it is enough to prove that $\eta\colon \Id \rtarr \bR\bH$ is monoidal. 
The unit diagram commutes since $\eta$ is the adjoint of the unit map of spectra $\epz \colon S_G \rtarr \bS_G(\bf{I})$ when $X=S^0$.
The other diagram we must show commutes is 
\[\xymatrix{
X \sma Y \ar[r]^-{\eta \sma \eta} \ar@{=}[d] & 
\bS_G(^{\bullet}\! X)_0 \sma \bS_G(^{\bullet}\! Y)_0\ar[d]^{\varphi_0} \\
 X \sma Y \ar[r]_-{\eta} & \bS_G(^{\bullet}\! (X\sma Y))_0.} \]
Since $\eta$ is the inclusion into the $0$-simplices of the bar construction, this is clear from the definition of $\varphi$. 
\end{proof}

\subsection{Ring, module, and algebra Eilenberg-Mac\,Lane $G$-spectra}

We show how the symmetric monoidal machine $\shim$  transports
$G$-rings, $G$-modules, and $G$-algebras to their genuine $G$-spectrum level analogs.

A commutative topological $G$-monoid $A$ is the same thing as a $\mathbf{Com}$-$G$-space, where 
$\mathbf{Com}$ is the commutativity operad; we require the basepoint $0$ of $A$ to be nondegenerate.   By \myref{PIvsF} and \myref{RRG}, the functor  $\bR$ 
from $G$-spaces to $\PI$-$G$-spaces given by  $(\bR A)({\mathbf n})= A^n$ takes $\mathbf{Com}$-$G$-spaces to $\sF$-$G$-spaces.   The sum induces the $\sF$-action, and $\bR A$ is trivially \gen-special.   

We restrict attention to discrete abelian $G$-groups $A$. Then $\shim A$
is an Eilenberg-Mac\,Lane $G$-spectrum $HA$.   We have the free abelian group functor $\bZ[-]$ 
that sends a $G$-set $S$ to the free abelian $G$-group  $\bZ[S]$. We view $\bZ[S]$ as obtained from 
the based version $\bZ[S_+]$ by setting the basepoint equal to zero. The composite $\bR \bZ[-]$ sends 
finite $G$-sets to \gen-special $\sF$-$G$-spaces.  Specializing the functor $^\bullet\! (-)$ 
to $G$-sets, it gives a functor from based $G$-sets to $\sF$-$G$-spaces.   We have the following observation.

\begin{lem} For based $G$-sets $S_+$, there is a natural map $h\colon ^\bullet\! (S_+) \rtarr \bR \bZ[S]$ of $\sF$-$G$-spaces.
\end{lem}
\begin{proof}
We have $^{n}\! (S_+) = (\coprod_n S)_+$.  The map $h$ sends the basepoint to $0$ and sends the $j$th copy of $S$ to the
set $S$ viewed as the generating set of the $j$th coordinate of the product $\bZ[S]^n$.  It is easily checked that this is a map of $\sF$-$G$-spaces.
\end{proof}

The letter $h$ is meant to indicate that $\bS_G h\colon  \SI^{\infty}_G S_+ \rtarr H\bZ[S]$ gives a machine built avatar 
of a specialization of the Hurewicz homomorphism.  

We shall prove the following result.

\begin{lem}   The functor $\bR$ from abelian $G$-groups to $\sF$-$G$-spaces is lax symmetric monoidal.
\end{lem}
\begin{proof}  Since ${\bf I} = ^\bullet\! S^0$, the map $h$ specialized to $S=\{1\}$ gives the required
unit map $\bf I\rtarr \bR \bZ$.  For abelian $G$-groups $A$ and $B$, we must construct a map of $\sF$-$G$-spaces
\[\bR A \sma \bR B \rtarr \bR(A\otimes B).\] By the universal property of Day convolution, this amounts to constructing a natural transformation of functors out of $\sF \times \sF$,
\begin{equation}\label{switch}
\bR A \barwedge \bR B \rtarr \bR (A \otimes B)\circ \sma.
\end{equation}
Note that the latter sends a pair $(\bf m,\bf n)$ to $(A\otimes B)^{mn}$, where $\bf m\sma \bf n = \bf{mn}$,
ordered lexicographically.  At $(\bf m,\bf n)$, the map in (\ref{switch}) is the composite \[A^m \times  B^n \rtarr (A\times B)^{mn} \rtarr (A\otimes B)^{mn},\]
where the first map uses the lexicographical ordering to send a pair $(\underline{a},\underline{b})$ to the $mn$-tuple that has $(a_i,b_j)$ in the $(i,j)$th position, and the second map is the 
$mn$-fold product of the canonical map $A\times B\rtarr A\otimes B$.
It is straightforward to check the required coherence relations.
\end{proof}

\begin{cor} The composite functor $\symfsma\com \bR$ from abelian $G$-groups to Eilenberg-Mac\,Lane $G$-spectra
takes $G$-rings, $G$-modules over $G$-rings, and $G$-algebras over commutative $G$-rings to ring $G$-spectra, 
module $G$-spectra over $G$-ring spectra, and algebras over commutative ring $G$-spectra.
\end{cor}
\begin{proof}
$G$-rings are just monoids in the symmetric monoidal category $(G\sA b,\otimes)$, and so on. 
Preservation properties such as these are formal for any lax symmetric monoidal functor.
\end{proof}

\begin{rem} In \cite{BO}, Bohmann and Osorno use an infinite loop space machine starting from 
Guillou and May \cite{GM2} to construct Eilenberg-Mac\,Lane $G$-spectra $HM$ for all Mackey
functors $M$, whereas the fixed points of abelian $G$-groups give only very special examples.      
However, their machine is not yet known to work multiplicatively.  There is work in progress 
seeking a common generalization.
\end{rem}

\subsection{Homotopies}

We give a brief discussion of how the Segal machine $\bS_G$ sees homotopies. 
The category of orthogonal $G$-spectra is tensored over the category of spaces via the half-smash product.  For a $G$-spectrum $E$ and 
a space $A$,  the half-smash product $E\sma A_+$ has $V$th $G$-space $(E\sma A_+)(V) = E(V)\sma A_+$.  We are mainly interested in $A=I$, and we may as well 
restrict to CW complexes $A$.  If $A$ is contractible and $E$ is an $\OM$-$G$-spectrum, then so is $E\sma A_+$.  
The category of $\sF$-$G$-spaces is also tensored over the category of spaces via the half-smash product.  
For an $\sF$-$G$-space $X$,  
$(X \sma A_+)({\bf n}) = X({\bf n})\sma A_+$.   Similarly, the category of $\sF_G$-$G$-spaces is tensored over spaces.
In these contexts, just as for $G$-spectra, homotopies between maps $X\rtarr Y$ are given by 
maps $X\sma I_+ \rtarr Y$.  

\begin{prop}  The Segal machine $\bS_G$ (in any of its avatars) preserves tensors with spaces and therefore preserves homotopies.
\end{prop}   
\begin{proof} Arguing as in \cite[Remark 3.3]{MMO} or just inspecting definitions, we see for example that 
$$   (B^{\bN_G}Y)(S^V))\sma A_+ \iso (B^{\bN_G}(Y\sma A_+))(S^V).  $$
This gives the proof for machine $\shimsma$.  The analogous commutation relation holds for our other machines.  
\end{proof}

\bibliographystyle{plain}
\bibliography{references}

\end{document}